\pdfoutput=1
\documentclass[letterpaper, 10 pt, conference]{ieeeconf}  

\IEEEoverridecommandlockouts                              

\usepackage{mathrsfs}
\usepackage[font={small}]{caption}
\usepackage{scalerel}
\usepackage{url}
\usepackage{bbold}
\usepackage{amsfonts}
\usepackage{amsmath,amssymb,amsfonts}
\usepackage{graphicx}
\usepackage{wrapfig}
\usepackage{mathrsfs} 
\usepackage{algorithm,algorithmic}
\usepackage{array}
\usepackage{times}
\usepackage{url}
\usepackage{cite}
\newcommand{\citet}[1]{\cite{#1}}
\usepackage{upgreek}
\usepackage{float}
\usepackage{longtable}
\usepackage{color}
\usepackage{wasysym}
\usepackage{grffile}
\usepackage{stackengine}
\usepackage[capitalize,noabbrev]{cleveref}

\allowdisplaybreaks

\newcommand{\ab}{\mathbf{a}}
\newcommand{\bb}{\mathbf{b}}

\newcommand{\eb}{\mathbf{e}}

\newcommand{\mb}{\mathbf{m}}

\newcommand{\ub}{\mathbf{u}}
\newcommand{\wb}{\mathbf{w}}
\newcommand{\xb}{\mathbf{x}}
\newcommand{\yb}{\mathbf{y}}
\newcommand{\zb}{\mathbf{z}}

\newcommand{\Ab}{\mathbf{A}}

\newcommand{\Gb}{\mathbf{G}}

\newcommand{\Ib}{\mathbf{I}}

\newcommand{\Pb}{\mathbf{P}}

\newcommand{\Rb}{\mathbf{R}}

\newcommand{\Vb}{\mathbf{V}}

\newcommand{\Xb}{\mathbf{X}}
\newcommand{\Yb}{\mathbf{Y}}

\newcommand{\Xc}{\mathcal{X}}

\newcommand{\argmin}{\text{argmin}}

\newcommand{\argmax}{\text{argmax}}

\newcommand{\norm}[1]{\left\lVert#1\right\rVert}

\newtheorem{theorem}{Theorem}

\newtheorem{assumption}{Assumption}

\newtheorem{corollary}[theorem]{Corollary}
\newtheorem{definition}{Definition}

\newtheorem{lemma}[theorem]{Lemma}

\newtheorem{remark}{Remark}

\title{\LARGE \bf Dynamic Regret Analysis of Safe Distributed Online Optimization for Convex and Non-convex Problems}

\author{Ting-Jui Chang, Sapana Chaudhary, Dileep Kalathil and Shahin Shahrampour  
\thanks{T.J. Chang and S. Shahrampour are with the Department of Mechanical and Industrial Engineering, Northeastern University, Boston, MA 02115, USA. 
{\tt\footnotesize email:\{chang.tin,s.shahrampour\}@northeastern.edu}.}%
\thanks{S. Chaudhary and D. Kalathil are with the Department of Electrical and Computer Engineering, Texas A\&M University, College Station, TX 77843, USA.
{\tt\footnotesize email:\{sapanac,dileep.kalathil\}@tamu.edu}.}%
}

\begin{document}

\maketitle
\thispagestyle{empty}
\pagestyle{empty}

\begin{abstract}
This paper addresses safe distributed online optimization over an unknown set of linear safety constraints. A network of agents aims at jointly minimizing a \emph{global}, \emph{time-varying} function, which is only partially observable to each individual agent. Therefore, agents must engage in {\it local} communications to generate a {\it safe} sequence of actions competitive with the best minimizer sequence in hindsight, and the gap between the two sequences is quantified via dynamic regret. We propose distributed safe online gradient descent (D-Safe-OGD) with an exploration phase, where all agents estimate the constraint parameters collaboratively to build estimated feasible sets, ensuring the action selection safety during the optimization phase. We prove that for convex functions, D-Safe-OGD achieves a dynamic regret bound of $O(T^{2/3}{\color{black} \sqrt{\log T}} + T^{1/3}C_T^*)$, where $C_T^*$ denotes the path-length of the best minimizer sequence. We further prove a dynamic regret bound of $O(T^{2/3}{\color{black} \sqrt{\log T}} + T^{2/3}C_T^*)$ for certain non-convex problems, which establishes the first dynamic regret bound for a {\it safe distributed} algorithm in the {\it non-convex} setting.

\end{abstract}
\section{Introduction}
Online learning (or optimization) is a sequential decision-making problem modeling a repeated game between a learner and an adversary \cite{hazan2016introduction}. At each round $t$, $t\in [T]\triangleq\{1,\ldots,T\}$, the learner chooses an action $\xb_t$ in a convex set $\Xc \subseteq \mathrm{R}^d$ using the information from previous observations and suffers the loss $f_t(\xb_t)$, where the function $f_t:\Xc\to\mathrm{R}$ is chosen by the adversary. Due to the sequential nature of the problem, a commonly used performance metric is \emph{regret}, defined as the difference between the cumulative loss incurred by the learner and that of a benchmark comparator sequence. When the comparator sequence is fixed, this metric is called {\it static} regret, defined as follows
\begin{align}\label{eq:static}
\mathbf{Reg}^s_T \triangleq \sum^T_{t=1}f_t(\xb_t)-\min_{\xb\in\Xc}\sum^T_{t=1}f_t(\xb).
\end{align}
Static regret is well-studied in the online optimization literature. In particular, it is well-known that online gradient descent (OGD) achieves a $O(\sqrt{T})$ (respectively, $O(\log T)$) static regret bound for convex (respectively,  exp-concave and strongly convex) problems \cite{zinkevich2003online,hazan2007logarithmic}, and these bounds are optimal in the sense of matching the lower bound of regret in the respective problems \cite{hazan2016introduction}. 

For non-convex problems, however, we expect that the standard regret notion used in the convex setting may not be a tractable measure for gauging the algorithm performance. For example, in the context of online non-convex optimization, \cite{hazan2017efficient} quantified the regret in terms of the norm of (projected) gradients, consistent with the stationarity measure in offline optimization. More recently, \citet{ghai2022non} showed that under certain geometric and smoothness conditions, OGD applied to non-convex functions is an approximation of online mirror descent (OMD) applied to convex functions under reparameterization. In view of this equivalence, OGD achieves a $O(T^{2/3})$ static regret defined in \eqref{eq:static}.

A more stringent benchmark for measuring the performance of online optimization algorithms is the \emph{dynamic} regret \cite{besbes2015non,jadbabaie2015online}, defined as
\begin{equation}\label{eq: definition of dynamic regret}
    \mathbf{Reg}^d_T \triangleq  \sum^T_{t=1}f_t(\xb_t)-\sum^T_{t=1}f_t(\xb_t^*),
\end{equation}
where $\xb_t^*\triangleq \argmin_{\xb\in\Xc}f_t(\xb)$. It is well-known that dynamic regret scales linearly with $T$ in the worst-case scenario, when the function sequence fluctuates drastically over time. Therefore, various works have adopted a number of variation measures to characterize the dynamic regret bound. We provide a detailed review of these measures in Section \ref{sec:lit} and describe the safe distributed online optimization in the next section.

\subsection{Safe Distributed Online Optimization}
There are two distinctive components that make ``safe distributed online optimization" more challenging than the standard online optimization:

(i) {\bf Distributed Setting:} Distributed online optimization has been widely applied to robot formation control \cite{dixit2019online}, 
distributed target tracking \cite{shahrampour2018distributed}, and localization in sensor networks \cite{akbari2015distributed}. In this setting, a network of $m$ agents aims at solving the online optimization problem collectively. The main challenge is that the time-varying function sequence is only partially observable to each individual agent. Each agent $j\in [m]$ receives (gradient) information about the ``local" function $f_{j,t}(\cdot)$, but the objective is to control the dynamic regret of each agent with respect to the global function $f_t(\cdot) = \sum_{i=1}^m f_{i,t}(\cdot)$, i.e.,  
\begin{align}\label{Eq: Individual regret}
    \mathbf{Reg}^d_{j,T} &\triangleq  \sum^T_{t=1} f_t(\xb_{j,t})-\sum^T_{t=1}f_t(\xb_t^*) \notag\\
    &=\sum^T_{t=1}\sum_{i=1}^m f_{i,t}(\xb_{j,t})-\sum^T_{t=1}f_t(\xb_t^*).
\end{align}
Therefore, despite availability of only local information, the action sequence of agent $j$ is evaluated in the global function, and is compared to the global minimizer sequence. It is evident that agents must communicate with each other (subject to a graph/network constraint) to approximate the global function, which is common to distributed problems. The network structure and communication protocols are provided in Section \ref{sec:net}.

(ii) {\bf Safety Constraints:} The literature on distributed online optimization has mostly focused on problems where the constraint set $\Xc$ is known, and less attention has been given to problems with {\it unknown} feasible sets (see Section \ref{sec:lit} for a comprehensive literature review). However, in many real-world applications, this  set represents certain safety constraints that are {\it unknown} in advance. Examples include voltage regulation constraints in power systems \cite{dobbe2020learning}, transmission bandwidth in communication networks due to human safety considerations \cite{luong2019applications} and stabilizing action set in robotics applications \cite{aastrom2010feedback}. In these scenarios, one needs to perform parameter estimation to {\it learn} the safety constraints while ensuring that the regret is as small as possible.

\subsection{Contributions}

In this work, we address the problem of {\it distributed} online optimization with {\it unknown} linear safety constraints. In particular, the constraint set $\mathcal{X}^s=\{\xb \in \mathrm{R}^d: \Ab \xb \leq \bb \}$ is linear, where $\Ab$ is {\it unknown} and must be learned by agents to subsequently choose their actions from this set. The superscript $s$ in $\mathcal{X}^s$ alludes to safety. Our specific objective is to study {\it dynamic} regret \eqref{eq: definition of dynamic regret} for both convex and non-convex problems when the set $\mathcal{X}^s$ is unknown to agents. Our contributions are three-fold:

\begin{itemize}
    \item[1)] We propose and analyze safe distributed online gradient descent (D-Safe-OGD) algorithm, which has two phases (exploration and optimization). In the exploration phase, agents individually explore and jointly estimate the constraint parameters in a distributed fashion. Then, each agent constructs a feasible set with its own estimate, which ensures the action selection safety with high probability (\cref{L: Bound for estimation error}). Since the estimates are only local, in the optimization phase, agents apply distributed OGD projected to {\it different} feasible sets, which brings forwards additional technical challenges.
    \item[2)] We analyze D-Safe-OGD in the {\it convex} setting. Due to the challenge discussed in the previous item, we cannot directly apply existing results on distributed online optimization with a common feasible set. We tackle this problem by leveraging the geometric property of linear constraints when agents' estimates are close enough to each other. Then, we strike a trade-off between the exploration and optimization phases and prove (in \cref{T: Dynamic regret bound}) a dynamic regret bound of $O(T^{2/3}{\color{black} \sqrt{\log T}}+T^{1/3}C_T^*)$, where
\begin{equation}\label{eq: path length start}
    C_T^*\triangleq \sum_{t=2}^T\norm{\xb^*_t-\xb^*_{t-1}},
\end{equation}
is the {\it path-length}, defined with respect to the global minimizer sequence \cite{mokhtari2016online,jadbabaie2015online}. If the problem is centralized (single agent) and the comparator sequence is fixed, i.e., $\xb^*_t=\xb$, our result recovers the static regret bound of \citet{chaudhary2021safe}.
    \item[3)] We further analyze D-Safe-OGD in the {\it non-convex} setting. We draw up on the idea of the algorithmic equivalence between ODG and OMD \citet{ghai2022non} to establish that  in certain problem geometries (Assumptions \ref{A: Geometric relation between phi and q}-\ref{A: Properties of mapping q}), D-Safe-OGD can be approximated by a distributed OMD algorithm applied to a reparameterized ``convex" problem. We prove that the dynamic regret is upper bounded by $O(T^{2/3}{\color{black} \sqrt{\log T}}+T^{2/3}C_T^*)$ in \cref{T: Dynamic regret bound (non-convex)}, which is the first dynamic regret bound for a {\it safe distributed} algorithm in the {\it non-convex} setting. If the problem is centralized (single agent) and the comparator sequence is fixed, i.e., $\xb^*_t=\xb$, our result recovers the static regret bound of \citet{ghai2022non}.
\end{itemize}
The proofs of our results are provided in the Appendix.

\section{Related Literature}\label{sec:lit}

\paragraph{I) Centralized Online Optimization:} For static regret, it is well-known that the optimal regret bound is $O(\sqrt{T})$ (respectively, $O(\log T)$) for convex (respectively, exp-concave and strongly convex) problems \cite{hazan2016introduction}. For dynamic regret, various regularity measures have been considered. A commonly used measure is the \emph{path-length} with respect to a general comparator sequence $\{\ub_t\}_{t=1}^T$, defined as
\begin{equation}\label{eq: path length}
    C_T(\ub_1,\ldots,\ub_T)\triangleq \sum_{t=2}^T\norm{\ub_t-\ub_{t-1}}.
\end{equation}
\citet{zinkevich2003online} first showed that for convex functions, OGD attains an upper bound of $O(\sqrt{T}(1+C_T))$ for dynamic regret, later improved to $O(\sqrt{T(1+C_T)})$ using expert advice \cite{zhang2018adaptive}. For strongly convex and smooth functions, \citet{mokhtari2016online} showed a regret bound of $O(C^*_T)$ for OGD. The notion of higher-order path-length $C_{p,T}^*\triangleq \sum_{t=2}^T\norm{\xb_t^*-\xb_{t-1}^*}^p$ has also been considered by several works. When the minimizer sequence $\{\xb^*_t\}_{t=1}^T$ is uniformly bounded, $O(C_{p,T}^*)$ implies $O(C_{q,T}^*)$ for $q<p$. \citet{zhang2017improved} proved that with multiple gradient queries per round, the dynamic regret is improved to $O(\min\{C^*_T,C^*_{2,T}\})$.

Other regularity measures include the function variation $V_T\triangleq \sum_{t=2}^T \sup_{\xb \in \Xc} |f_t(\xb)-f_{t-1}(\xb)|$ \cite{besbes2015non}, the predictive path-length $C^\prime_T(\ub_1,\ldots,\ub_T)\triangleq\sum_{t=2}^T\norm{\ub_t - \Psi_t(\ub_{t-1})}$ \cite{hall2013dynamical}, where $\Psi_t$ is a given dynamics, and the gradient variation $D_T\triangleq \sum_{t=1}^T \norm{\nabla f_t(\xb_t)-\mb_t}^2$ \cite{rakhlin2013online}, where $\mb_t$ is a predictable sequence computed by the learner. \citet{besbes2015non} proposed a restarting OGD and showed that when the learner only has access to noisy gradients, the expected dynamic regret is bounded by $O(T^{2/3}(V_T+1)^{1/3})$ and $O(\sqrt{T(1+V_T)})$ for convex and strongly convex functions, respectively. The above measures are not directly comparable to each other. In this regard, \citet{jadbabaie2015online} provided a dynamic regret bound in terms of $C^*_T$, $D_T$ and $V_T$ for the adaptive optimistic OMD algorithm. Also, \cite{chang2021online} revisited OGD with multiple queries in the unconstrained setup and established the regret bound of $O(\min \{V_T,C^*_T, C^*_{2,T}\})$ for strongly convex and smooth functions. Dynamic regret has also been studied for functions with a parameterizable structure \cite{ravier2019prediction} as well as composite convex functions \cite{ajalloeian2020inexact}. 

\paragraph{II) Distributed Online Optimization:}
\citet{yan2012distributed} studied distributed OGD for online optimization and proved a static regret bound of $O(\sqrt{T})$ (respectively, $O(\log T)$) for convex (respectively, strongly-convex) functions. Distributed online optimization for time-varying network structures was then considered in \cite{mateos2014distributed,akbari2015distributed,hosseini2016online}. \citet{shahrampour2018distributed} proposed a distributed OMD algorithm with a dynamic regret bound of $O(\sqrt{T}(1+C^*_T))$. \citet{dixit2019distributed} considered time-varying network structures and showed that distributed proximal OGD achieves a dynamic regret of $O(\log T(1+C^*_T))$ for strongly convex functions. \citet{zhang2019distributed} developed a method based on gradient tracking and derived a regret bound in terms of $C^*_T$ and a gradient path-length. More recently, \cite{eshraghi2022improving} showed that dynamic regret for strongly convex and smooth functions can be improved to $O(1+C^*_T)$ using both primal and dual information boosted with multiple consensus steps. The non-convex case is also recently studied in \cite{lu2021online}, where the regret is characterized by first-order optimality condition. 

Nevertheless, the works mentioned in {\bf (I)} and {\bf (II)} do not consider neither long-term nor safety constraints, which are discussed next.

\begin{table*}[h!]
\caption{Related works on centralized and distributed constrained online optimization for general functions with regret and constraint violation (CV) guarantees. Let $g(\xb) = (g_1(\xb), g_2(\xb),\dots,g_n(\xb))^\top$ is vector formed by $n$ convex constraints. Let $\beta \in (0,1)$, $\alpha_0 > 1$ and $[a]_{+} = \max\{0,a\}$. Problem type `C' stands for centralized, `D' stands for distributed (or decentralized), `CNX' stands for convex cost functions, and `N-CNX' stands for non-convex cost functions. Notes: $\dagger:$ The CV bound in \cite{yu2020low} can be further reduced to $O(1)$ under a Slater's condition assumption.}
\centering
\resizebox{2.05\columnwidth}{!}{\begin{tabular}{||c c c c c c|| } 
 \hline
 CV Type  & Problem & Reference & Regret Type & Regret Bound & CV Bound\\ 
 \hline
 \hline
 $\sum_{t=1}^T g_i\left(\xb_t\right)~\forall i\in[n]$  & C, CNX & \cite{mahdavi2012trading} & Static & $O(\sqrt{T})$ & $O(T^{3/4})$\\
 $\sum_{t=1}^T \max_{i\in[n]} g_i\left(\xb_t\right) $  & C, CNX & \cite{jenatton2016adaptive} & Static & $O(T^{\max\{\beta, 1-\beta\}})$ & $O(T^{1-\beta/2})$\\
 $\sum_{t=1}^T\left[g_i\left(\xb_t\right)\right]_{+}~\forall i\in[n]$  & C, CNX & \cite{yuan2018online} & Static & $O(T^{\max\{\beta, 1-\beta\}})$ & $O(T^{1-\beta/2})$\\
 $ \sum_{t=1}^T\|\left[g\left(\xb_t\right)\right]_+\|$  & C, CNX & \cite{yi2021regret} & Static & $O(T^{\max\{\beta, 1-\beta\}})$ & $O(T^{(1-\beta)/2})$\\
 $ \sum_{t=1}^T\|\left[g\left(\xb_t\right)\right]_+\|$  & C, CNX & \cite{yi2021regret} & Dynamic & $O(\sqrt{T(1+C_T)})$ & $O(\sqrt{T})$\\
 $\sum_{t=1}^T g_i\left(\xb_t\right)~\forall i\in[n]$  & C, CNX & \cite{yu2020low} & Static & $O(\sqrt{T})$ & $O(T^{1/4})^{\dagger}$\\
 $\left[\sum_{t=1}^T g_i\left(\xb_{j, t}\right)\right]_{+},~i=1,\forall j\in[m]$ & D, CNX & \cite{yuan2017adaptive} & Static & $O(T^{1/2+\beta/2})$ & $O(T^{1-\beta/4})$\\
 $\sum_{t=1}^T \sum_{j=1}^m \sum_{i=1}^n\left[\ab_i^\top\mathbf{\xb}_{j,t}\right]_{+}$ & D, CNX & \cite{yuan2020distributed} & Static & $O(\sqrt{T})$ & $O(T^{3/4})$\\
 $\sum_{t=1}^T \sum_{j=1}^m \sum_{i=1}^n\left[g_i\left(\mathbf{\xb}_{j,t}\right)\right]_{+}$  & D, CNX & \cite{yuan2021distributed} & Static & $O(T^{\max\{\beta, 1-\beta\}})$ & $O(T^{1-\beta/2})$\\
 $\frac{1}{m} \sum_{i=1}^m \sum_{t=1}^T\left\|\left[g_t\left(\xb_{i, t}\right)\right]_{+}\right\|$ & D, CNX & \cite{yi2022regret} & {Dynamic} & $O((\alpha_0^2 T^{1-\beta} + T^\beta(1+C_T))/\alpha_0)$ & $O(\sqrt{(\alpha_0+1)T^{2-\beta}})$\\
  $\sum_{t=1}^T \|\Ab \xb_{t} - \bb\| $ & C, CNX & \cite{chaudhary2021safe} & Static & $O(\sqrt{\log(T)}T^{2/3})$ & 0\\
  \hline
  $\sum_{t=1}^T \sum_{j=1}^m \|\Ab \xb_{j,t} - \bb\| $ & D, CNX & \textbf{This work} & Dynamic & $O\left(T^{2 / 3}{\color{black} \sqrt{\log T}}+T^{1 / 3} C_T^*\right)$ & 0\\
  $\sum_{t=1}^T \sum_{j=1}^m \|\Ab \xb_{j,t} - \bb\| $ & D, N-CNX & \textbf{This work} & Dynamic & $O\left(T^{2 / 3}{\color{black} \sqrt{\log T}}+T^{2 / 3} C_T^*\right)$ & 0\\
 \hline\hline
\end{tabular}}
\label{table:Dynamic regret bounds of related works 1}
\end{table*}

\paragraph{III) Constrained Online Optimization:}
Practical systems come with inherent system imposed constraints on the decision variable. Some examples of such constraints are inventory/budget constraints in one-way trading problem \cite{lin2019competitive} and time coupling constraints in networked distributed energy systems \cite{fan2020online}. For a known constraint set, projecting the decisions back to the constraint set is a natural way to incorporate constraints in OCO with projected OGD for general cost functions achieving $O(\sqrt{T})$ static regret. However, for complex constraints, projection can induce computational burden. An early work by \cite{hazan2012projection} in this space, solves the constrained optimization problem by replacing the quadratic convex program with simpler linear program using Frank-Wolfe. For understanding the following references better, let $\mathcal{X} = \{\xb \in \mathrm{X} \subseteq \mathrm{R}^d: g(\xb) \leq \mathbf{0}_n\}$ where $g(\xb) = (g_1(\xb), g_2(\xb),\dots,g_n(\xb))^\top$ is vector formed by $n$ convex constraints, and $\mathrm{X}$ is a closed convex set. The work by \cite{mahdavi2012trading} proposed to use a simpler closed form projection in place of true desired projection attaining $O(\sqrt{T})$ static regret with $\sum_{t=1}^T g_i\left(\xb_t\right) \leq O(T^{3/4})$ constraint violation $\forall i\in[n]$. Thus, their method achieves optimal regret with lesser computation burden at the cost incurring constraint violations. The follow up work by \cite{jenatton2016adaptive}, proposed adaptive step size variant of \cite{mahdavi2012trading} with $O\left(T^{\max \{\beta, 1-\beta\}}\right)$ static regret and $O\left(T^{ 1-\beta/2}\right)$ constraint violation for $\beta \in (0,1)$. These bounds were further improved in \cite{yu2020low} with a static regret bound of $O(\sqrt{T})$ and constraint violation bound of $O(T^{1/4})$. Here, the constraint violation is further reduced to $O(1)$ when $g_i(\xb)$ satisfy Slater condition. {The work by \cite{yuan2018online} considers stricter `cumulative' constraint violations of the form $
\sum_{t=1}^T\left[g_i\left(\xb_t\right)\right]_+~\forall i\in[n]$ and proposes algorithms with $O\left(T^{\max \{\beta, 1-\beta\}}\right)$ static regret and $O\left(T^{ 1-\beta/2}\right)$ `cumulative' constraint violation for $\beta \in (0,1)$.} For strongly convex functions, \cite{yuan2018online} gives $O(\log(T))$ static regret and constraint violations of respective forms is $O(\sqrt{\log(T)T})$. 
More recently, the work by \cite{yi2021regret} proposed an algorithm with $O\left(T^{\max \{\beta, 1-\beta\}}\right)$ regret and { $\sum_{t=1}^T\left\|\left[g\left(\xb_t\right)\right]_{+}\right\| \leq O\left(T^{(1-\beta)/2}\right)$} `cumulative' constraint violation. For strongly convex functions, \cite{yi2021regret} reduce both static regret and constraint violation bounds to $O(\log(T))$. Further, \cite{yi2021regret} give dynamic bounds of $O(\sqrt{T(1+C_T)})$ for regret and $O(\sqrt{T})$ `cumulative' constraint violation. The algorithms in \cite{mahdavi2012trading, jenatton2016adaptive, yu2020low, yuan2018online, yi2021regret} employ some flavour of online primal-dual algorithms. A series of recent works \cite{sun2017safety, chen2017online, neely2017online, yu2017online, cao2018online, liu2022simultaneously} have also dealt with time-varying constraints. \cite{yu2017online} specifically works with `stochastic' time varying constraints. 

More recently works in \cite{yuan2017adaptive, yuan2020distributed, yuan2021distributed, yi2022regret} have looked at distributed OCO with long term constraints. The work by \cite{yuan2017adaptive} proposes a consensus based primal-dual sub-gradient algorithm with $O(T^{1/2 + \beta_0})$ regret and $O(T^{1-\beta_0/2})$ constraint violation for $\beta_0 \in (0,0.5)$. Single constraint function is considered in \cite{yuan2017adaptive} and constraint violation is of the form $[\sum_{t=1}^T g_i\left(\xb_{j, t}\right)]_{+},~i=1,\forall j\in[m]$. In \cite{yuan2020distributed}, the authors propose algorithms for distributed online linear regression with $O(\sqrt{T})$ regret and $O(T^{3/4})$ constraint violation. Here, constraint violation takes the form $\sum_{t=1}^T \sum_{j=1}^m \sum_{i=1}^n\left[\ab_i^\top\mathbf{\xb}_{j,t}\right]_{+}$, where $\ab_i, i\in[n]$ is a constraint vector. Another primal-dual algorithm is presented in \cite{yuan2021distributed} with $O(T^{\max\{1-\beta, \beta\}})$ regret and $O(T^{1-\beta/2})$ constraint violation of the form $\sum_{t=1}^T \sum_{j=1}^m \sum_{i=1}^n\left[g_i\left(\mathbf{\xb}_{j,t}\right)\right]_{+}$ for $\beta \in (0,1)$. In all of \cite{yuan2017adaptive, yuan2020distributed, yuan2021distributed} constraint functions are known a priori. Most recently, \cite{yi2022regret} propose algorithms for distributed OCO with time-varying constraints, and for stricter `network' constraint violation metric of the form $\frac{1}{m} \sum_{i=1}^m \sum_{t=1}^T\left\|\left[g_t\left(\xb_{i, t}\right)\right]_{+}\right\|$. The algorithm in \cite{yi2022regret} gives a dynamic regret of $O((\alpha_0^2 T^{1-\beta} + T^\beta(1+C_T))/\alpha_0)$ with $O(\sqrt{(\alpha_0+1)T^{2-\beta}})$ constraint violation for $\alpha_0>1$ and $\beta \in (0,1)$. Additionally, constrained distributed OCO with coupled inequality constraints is considered in \cite{yi2020distributed, yi2021distributed}; with bandit feedback on cost function is considered in \cite{li2020online}; with partial constraint feedback in studied in \cite{sharma2021distributed}. For more references in this problem space, please see the survey in \cite{li2022survey}. 

\paragraph{IV) Safe Online Optimization:} Safe optimization is a fairly nascent field with only a few works studying per-time safety in optimization problems. In \cite{amani2019linear, khezeli2020safe} study the problem of safe linear bandits giving $O(\log(T)\sqrt{T})$ regret with no constraint violation, albeit under an assumption that a lower bound on the distance between the optimal action and safe set's boundary is known. Without knowledge such a lower bound, \cite{amani2019linear} show $O(\log(T)T^{2/3})$ regret. Safe convex and non-convex optimization is studied in \cite{usmanova2019safe, fereydounian2020safe}.
Safety in the context of online convex optimization is studied in \cite{chaudhary2021safe} with a regret of {$O(\sqrt{\log(T)}T^{2/3})$}. 

\begin{remark}
Different from the works listed above, we study the problem of safe \textit{distributed} online optimization with unknown linear constraints. We consider both convex and \textit{non-convex} cost functions. 
\end{remark}
\section{Preliminaries}
\subsection{Notation}
{\small
\begin{tabular}{|c||l|}
    \hline
    $[m]$ & The set $\{1,2,\ldots,m\}$ for any integer $m$ \\
    \hline
    $\norm{\cdot}_F$ & Frobenius norm of a matrix\\
    \hline
    $\norm{\cdot}_{\Vb}$ & $\sqrt{\xb^{\top}\Vb\xb}$, for $\xb\in \mathrm{R}^d$ and a positive-definite $\Vb$\\
    \hline
    $\Pi_{\Xc}[\cdot]$ & The operator for the projection to set $\Xc$ \\
    \hline
    $[\Ab]_{ij}$ & The entry in the $i$-th row and $j$-th column of $\Ab$\\
    \hline
    $[\Ab]_{:,j}$ & The $j$-th column of $\Ab$\\
    \hline
    $\mathbf{1}$ & The vector of all ones\\
    \hline
    $\eb_i$ & The $i$-th basis vector\\
    \hline
    $J_f(\xb)$ & The Jacobian of a mapping $f(\cdot)$ at $\xb$\\
    \hline
\end{tabular}}
\subsection{Strong Convexity and Bregman Divergence}
\begin{definition}
A function $f:\Xc\to \mathrm{R}$ is $\mu$-strongly convex ($\mu>0$) over the convex set $\Xc$ if $$f(\xb)\geq f(\yb)+\nabla f(\yb)^\top(\xb-\yb) +\frac{\mu}{2}\norm{\xb-\yb}^2,\: \forall\xb,\yb \in \Xc.$$
\end{definition}
\begin{definition}
For a strongly convex function $\phi(\cdot)$, the Bregman divergence w.r.t. $\phi(\cdot)$ over $\Xc$ is defined as
$$\mathcal{D}_{\phi}(\xb,\yb)\triangleq \phi(\xb) - \phi(\yb) - \nabla \phi(\yb)^{\top}(\xb - \yb), \xb,\yb\in\Xc.$$
\end{definition}

\subsection{Network Structure}\label{sec:net}
The underlying network topology is governed by a symmetric doubly stochastic matrix $\Pb$, i.e., $[\Pb]_{ij}\geq 0, \forall i,j \in [m]$, and each row (or column) is summed to one. If $[\Pb]_{ij}> 0$, agents $i$ and $j$ are considered neighbors, and agent $i$ assigns the weight $[\Pb]_{ij}$  to agent $j$ when they communicate with each other. We assume that the graph structure captured by $\Pb$ is connected, i.e., there is a (potentially multi-hop) path from any agent $i$ to another agent $j\neq i$. Each agent is considered as a neighbor of itself, i.e., $[\Pb]_{ii}> 0$ for any $i\in [m]$. These constraints on the communication protocol implies a geometric mixing bound for $\Pb$ \cite{liu2008monte}, such that $\sum_{j=1}^m \left|[\Pb^k]_{ji}-1/m\right|\leq \sqrt{m}\beta^k,$ for any $i\in [m]$,
where $\beta$ is the second largest singular value of $\Pb$. For a connected graph, $D_G$, the graph diameter, is defined as the number of the links of the longest of the shortest paths connecting any two agents ($D_G$ is at most $m-1$).

\begin{remark}
In all of the algorithms proposed in the paper, we will see $\Pb$ as an input. This does not contradict the decentralized nature of the algorithms, as agent $i$ only requires the knowledge of $[\Pb]_{ji}>0$ for any $j$ in its neighborhood. The knowledge of $\Pb$ is not global, and each agent only has local information about it.
\end{remark}
\section{Safe Set Estimation}
In our problem setup, the algorithm receives noisy observations of the form $\hat{\xb}_{i,t}= \Ab\xb_{i,t} + \wb_{i,t}~\forall i\in[m]$ at every timestep $t$, where the nature of noise $\wb_{i,t}$ is described below. Here $\Ab \in \mathrm{R}^{n \times d}, \bb \in \mathrm{R}^{n}$ and $n$ is the number of constraints. Note that all agent updates are synchronous. 

\subsection{Assumptions}
To study the problem of safe distributed online optimization, we make the following assumptions common to both the convex and the non-convex problem settings. 
\begin{assumption}\label{A1} The set  $\mathcal{X}^s$ is a closed polytope, hence,  convex and compact. Also, $\norm{\xb} \leq L, \forall \xb \in \mathcal{X}^s$, and $\max_{i \in [n]} \norm{\ab_{i}}_{2} \leq L_{A}$, where $\ab_{i}$ denotes the $i$-th row of $\Ab$. 
\end{assumption}

\begin{assumption}\label{A2} The constraint noise sequence $\{\wb_{i,t}, t \in [T] \}$ is $R$-sub-Gaussian with respect to a filtration $\{\mathcal{F}_{i,t}, t \in [T]\}$, i.e., $\forall t \in [T], \forall i \in [m]$, $\mathbb{E}[\wb_{i,t}|\mathcal{F}_{i,t-1}] = 0$ and for any $\xb\in\mathrm{R}^d$, $\langle\xb, \wb_{i,t}\rangle$ is $R\norm{\xb}$ sub-Gaussian:
\begin{equation*}
    \mathbb{E}[\exp(\sigma \xb^{\top}\wb_{i,t})\,|\,\mathcal{F}_{t-1}] \leq \exp(\sigma^2 (R\norm{\xb})^2/2).   
\end{equation*}
\end{assumption}

\begin{assumption}\label{A3}
    Every agent has knowledge of a \textit{safe baseline action} $\xb^{s} \in \mathcal{X}^{s}$ such that $\Ab \xb^{s}  =\bb^{s} < \bb$. The agents are aware of $\xb^{s}$ and $\bb^{s}$ and thus, the safety gap $\Delta^{s} = \min_{i\in [n]} (b_{i} - b^{s}_{i})$, where $b_i$ (respectively, $b^s_i$) denotes the $i$-th element of $\bb$ (respectively, $\bb^s$).
\end{assumption}

The first two assumptions are typical to online optimization. The third assumption stems from the requirement to be absolutely safe at every time step. The assumption warrants need for a safe starting point which is readily available in most practical problems of interest. Similar assumptions can be found in previous literature on safe linear bandits \cite{amani2019linear, khezeli2020safe}, safe convex and non-convex optimization \cite{usmanova2019safe, fereydounian2020safe}, and safe online convex optimization \cite{chaudhary2021safe}. 

\subsection{Explore and Estimate}
In this section we present an algorithmic subroutine, Alg. \ref{alg:Constraint estimation}, for agents to obtain sufficiently good local estimates of $\mathcal{X}^s$ before beginning to perform OGD. For the first $T_0$ timesteps, each agent safely explores around the baseline action $\xb^s$. Each exploratory action is a $\gamma$-weighted combination of the baseline action and an i.i.d random vector $\zeta_{i,t}$. Here, for the agent $i \in [m]$ at timestep $t \in [T_0]$, $\gamma \in [0,1)$ and $\zeta_{i,t}$ is zero mean i.i.d random vector with $\|\zeta_{i,t}\| \leq L$ and $Cov(\zeta_{i,t}) = \sigma_{\zeta}^2 \Ib.$ Performing exploration in this manner ensures per time step safety requirement as noted in the Lemma \ref{L:safe_exploration}. 

\begin{lemma}\label{L:safe_exploration}
    Let Assumptions \ref{A1} and \ref{A3} hold. With $\gamma=\frac{\Delta^s}{LL_A}$, $\Ab \xb_{i,t} \leq \bb$ for each $\xb_{i,t}= (1-\gamma)\xb^s + \gamma\zeta_{i,t}~\forall i \in [1,m], t \in [1,T_0]$.  
\end{lemma}

Once the exploration phase is finished, agents construct local functions $l_i(\Ab)$ of the form $l_i(\Ab)\triangleq\sum_{t=1}^{T_0}\norm{\Ab\xb_{i,t} - \hat{\xb}_{i,t}}^2 + \frac{\lambda}{m}\norm{\Ab}_F^2.$ Then, for timesteps $t \in [T_0+1, T_0+T_1]$, Alg. \textbf{EXTRA} \cite{shi2015extra} is used to solve the global Least Squares (LS) estimation problem $\sum_{i=1}^m l_i(\Ab)$ in a distributed fashion. $\alpha$ is chosen based on EXTRA. 

\begin{lemma}\label{L: Bound for estimation error}
Suppose Assumptions \ref{A1} and \ref{A2} hold. By running Algorithm \ref{alg:Constraint estimation} with $T_0 = \Theta(\frac{L^2}{m\gamma^2\sigma_{\zeta}^2}\log (\frac{d}{\delta}))$ for data collection and $T_1=O(\log T^{\rho})$ where $\rho$ is a positive constant, then with probability at least $(1-2\delta)$, we have $\forall k \in [n]$ and $\forall i,j \in[m]$
\begin{equation}
\begin{split}
    \norm{\hat{\ab}^i_{k}-\ab_k} &\leq \frac{1}{T^{\rho}} + \frac{R\sqrt{d\log \big(\frac{1+mT_0L^2/\lambda}{\delta/n}\big)} + \sqrt{\lambda}L_A}{\sqrt{m\gamma^2\sigma_{\zeta}^2T_0}},\\
    \norm{\hat{\ab}^i_{k}-\hat{\ab}^j_{k}} &\leq \frac{2}{T^{\rho}},
\end{split}
\end{equation}
where $\hat{\ab}^i_{k}$ and $\ab_k$ are the $k$-th rows of agent $i$'s estimate and the true parameters $\Ab$, respectively.
\end{lemma}

\begin{algorithm}[!h]
\caption{Distributed Constraint Parameter Estimation}
\label{alg:Constraint estimation}
\begin{algorithmic}[1]
    \STATE {\bf Require:} number of agents $m$, doubly stochastic matrix $\Pb\in \mathrm{R}^{m\times m}$, $\Tilde{\Pb}\triangleq\frac{\Ib+\Pb}{2}$, hyper-parameters $\alpha$, $\gamma$ and $\lambda$, data-collection duration $T_0$, constraint-estimation duration $T_1$, a strictly feasible point $\xb^s$.
    \STATE {\bf Explore around baseline action:}
    \FOR{$t=1,2,\ldots,T_0$}
        \FOR{$i=1,2,\ldots,m$} 
            \STATE Select action $\xb_{i,t}= (1-\gamma)\xb^s + \gamma\zeta_{i,t}$
            \STATE Receive noisy observation $\hat{\xb}_{i,t}= \Ab\xb_{i,t} + \wb_{i,t}$
        \ENDFOR   
    \ENDFOR
    \STATE {\bf Form local functions using explored data:}
    \begin{equation*}
      l_i(\Ab)\triangleq\sum_{t=1}^{T_0}\norm{\Ab\xb_{i,t} - \hat{\xb}_{i,t}}^2 + \frac{\lambda}{m}\norm{\Ab}_F^2.
    \end{equation*}
    \STATE \textbf{Use EXTRA} \cite{shi2015extra} \textbf{to solve global LS problem $\sum_{i=1}^m l_i(\Ab)$:}
    \STATE Randomly generate $\widehat{\Ab}^0_i \in \mathrm{R}^{n \times d}$ for all $i\in [m]$. 
    \STATE $\forall i\in [m],\:\widehat{\Ab}^{T_0+1}_i = \sum_{j=1}^m [\Pb]_{ji}\widehat{\Ab}^0_j - \alpha\nabla l_i(\widehat{\Ab}^0_i).$
    \FOR{$t=T_0,\ldots,T_0+T_1-2$}
        \FOR{$i=1,2,\ldots,m$}
            \STATE $\widehat{\Ab}^{t+2}_i = \sum_{j=1}^m 2 [\Tilde{\Pb}]_{ji}\widehat{\Ab}^{t+1}_j - \sum_{j=1}^m [\Tilde{\Pb}]_{ji}\widehat{\Ab}^{t}_j - \alpha[\nabla l_i(\widehat{\Ab}^{t+1}_i)-\nabla l_i(\widehat{\Ab}^t_i)].$
        \ENDFOR
    \ENDFOR
\end{algorithmic}
\end{algorithm}

Let us now define the estimated safe sets for each agent $i \in [m]$. Let the parameter estimate for agent $i \in [m]$ returned by \textbf{EXTRA} at the end of $T_0+T_1$ timestep be denoted by $\widehat{\Ab}_{i}$. For each row $k \in [n]$ of $\widehat{\Ab}_{i}$, a ball centered at $\hat{\ab}_{k}^i$ with a radius of $\mathcal{B}_r$ can be defined as follows
\begin{equation}\label{Eq: The ball set w.r.t. an estimate}
    \mathcal{C}_{i,k}\triangleq\{\ab_k\in \mathbb{R}^d:\norm{\ab_k - \hat{\ab}_{k}^i} \leq \mathcal{B}_r\},
\end{equation}
where $\mathcal{B}_r = \frac{1}{T^{\rho}} + \frac{R\sqrt{d\log \big(\frac{1+mT_0L^2/\lambda}{\delta/n}\big)} + \sqrt{\lambda}L_A}{\sqrt{\frac{1}{2}m\gamma^2\sigma_{\zeta}^2T_0}}$. The true parameter $\ab_k$ lies inside the set $\mathcal{C}_{i,k}$ with a high probability of $(1 - \delta/n)$. 
Now, using \eqref{Eq: The ball set w.r.t. an estimate}, safe estimated set for agent $i \in [m]$ can be constructed as follows:
\begin{equation}\label{Eq: Agent's feasible set}
    \widehat{\mathcal{X}}^s_i\triangleq\{\xb\in\mathbb{R}^d: \Tilde{\ab}^{\top}_k\xb\leq \bb_k,\;\forall \Tilde{\ab}_k \in \mathcal{C}_{i,k},\;\forall k \in [n]\}.
\end{equation}


\section{Dynamic Regret Analysis for the Convex Setting}
During the $T_0+T_1$ steps in \cref{alg:Constraint estimation} agents do not expend any effort to minimize the regret. This is due to the fact that without the knowledge of the feasible set, they cannot perform any projection. In this section, we propose D-Safe-OGD, which allows agents to carry out a safe distributed online optimization, and we analyze D-Safe-OGD in the \emph{convex} setting.

D-Safe-OGD is summarized in Algorithm \ref{alg:D-safe-ogd-linear-constraint}, where in the exploration phase, all agents collaboratively estimate the constraint parameters based on Algorithm \ref{alg:Constraint estimation}, and then each agent constructs the feasible set based on its own estimate. In the optimization phase, the network applies distributed OGD, where all agents first perform gradient descent with their local gradients, and then they communicate their iterates with neighbors based on the network topology imposed by $\Pb$. We note that the projection operator of each agent is defined w.r.t. the local estimated feasible set (line 8 of \cref{alg:D-safe-ogd-linear-constraint}), thereby making the feasible sets close enough but slightly different from each other. Therefore, previous regret bounds for distributed online optimization over a common feasible set (e.g., \cite{yan2012distributed,hosseini2016online,shahrampour2018distributed,eshraghi2022improving}) are not immediately applicable. We tackle this challenge using the geometric property of linear constraints \cite{fereydounian2020safe}, and we present an upper bound on the dynamic regret in terms of the path-length regularity measure.

We adhere to the following standard assumption in the context of OCO:  
\begin{assumption}\label{A4}
    The  cost functions $f_{i,t}$ are convex  $\forall i \in [m]$ and $\forall t \in [T]$, and they have a bounded gradient, i.e., $\norm{\nabla f_{i,t}(\xb)} \leq G$ for any  $\xb\in \mathcal{X}^s$.
\end{assumption}

\begin{theorem}\label{T: Dynamic regret bound}
Suppose Assumptions \ref{A1}-\ref{A3} and \ref{A4} hold and $T=\Omega\left(\big(\frac{L^2}{m\gamma^2\sigma_{\zeta}^2}\log (\frac{d}{\delta})\big)^{3/2}\right)$. By running Algorithm \ref{alg:D-safe-ogd-linear-constraint} with $\gamma\leq \frac{\Delta^s}{LL_A}$, $\eta=\Theta(T^{-1/3})$, $T_0 = \Theta(T^{2/3})$ and $T_1 = \Theta(\log T^{\rho})$, we have with probability at least $(1-2\delta)$
\begin{equation*}
    \xb_{i,t}\in \mathcal{X}^s,\;\forall i\in [m], t\in[T],~~~\text{and}
\end{equation*}
\begin{equation*}
    \mathbf{Reg}^d_{i,T} = O\left({\color{black} \frac{\beta}{\left(1-\beta\right)}T^{2/3}\sqrt{\log T}} + T^{1/3}C^*_T\right),\;\forall i \in [m].  
\end{equation*}
\end{theorem}
\cref{T: Dynamic regret bound} establishes a dynamic regret bound for D-Safe-OGD that is at least $O(T^{2/3}{\color{black} \sqrt{\log T}})$, and for a large enough path-length the bound becomes $O(T^{1/3}C^*_T)$. We can also see the impact of network topology through $\beta$, the second largest singular value of $\Pb$. When the network connectivity is stronger (i.e., $\beta$ is smaller), the regret bound is tighter.

\begin{corollary}
Suppose that the comparator sequence is fixed over time, i.e., $\xb^*_t = \xb^*,\; \forall t\in [T]$. Then, the individual regret bound is $O(T^{2/3})$, which recovers the static regret bound of the centralized case in \cite{chaudhary2021safe} in terms of order.
\end{corollary}

\begin{remark}
Note that when $\Ab$ is known, there is no estimation error, and the trade-off in terms of $\eta$ and $T_0$ no longer exists. In other words, the agents do not incur the initial regret of $T_0+T_1=O(T^{2/3})$, caused by estimation. Then, by choosing $\eta=\Theta(\frac{1}{\sqrt{T}})$, the resulting bound is $O(\sqrt{T}(1+C^*_T)$, which recovers the result of \cite{shahrampour2018distributed} in the order sense.
\end{remark}


\begin{algorithm}[!h]
\caption{Distributed Safe OGD with linear constraints}
\label{alg:D-safe-ogd-linear-constraint}
\begin{algorithmic}[1]
    \STATE {\bf Require:} number of agents $m$, doubly stochastic matrix $\Pb\in \mathrm{R}^{m\times m}$, hyper-parameters $\gamma,\eta,\delta,\lambda$, time horizon $T$, a strictly feasible point $\xb^s$.
    \STATE Specify $T_0$ and $T_1$ based on given hyper-parameters and run Algorithm \ref{alg:Constraint estimation} to learn agents estimates $\{\hat{\Ab}_i\}_{i\in[m]}$ in a distributed fashion.
    \STATE For all $i\in[m]$, construct the safe set $\widehat{\mathcal{X}}^s_i$ from the estimate $\widehat{\Ab}_i$.
    \STATE {\bf Distributed online gradient descent over different feasible sets:}
    \STATE Let $T_s\triangleq(T_0+T_1+1)$.
    \FOR{$t=T_s,\ldots,T$}
        \FOR{$i=1,2,\ldots,m$}
         \STATE   \begin{equation*}
                \yb_{i,t} = \Pi_{\widehat{\mathcal{X}}^s_i}\left[\xb_{i,t} - \eta\nabla f_{i,t}(\xb_{i,t})\right].    
            \end{equation*}
        \ENDFOR
        \STATE For all $i\in[m]$, 
            \STATE
            \begin{equation*}
                \xb_{i,t+1} = \sum\limits_{j=1}^m[\Pb^{}]_{ji}\yb_{j,t}.    
            \end{equation*}
    \ENDFOR
\end{algorithmic}
\end{algorithm}

\section{Dynamic Regret Analysis for the Non-convex Setting}
 In this section, we study the \emph{non-convex} setting for safe distributed online optimization. Even for offline optimization in the non-convex setting, the standard metric for the convergence analysis is often stationarity, i.e., characterizing the decay rate of the gradient norm. In online optimization, we can also expect that standard regret notions, used in the convex setting, may not be tractable for understanding the algorithm performance. However, in a recent work by \cite{ghai2022non}, the authors studied an algorithmic equivalence property between OGD and OMD for certain problem geometries (Assumptions \ref{A: Geometric relation between phi and q}-\ref{A: Properties of mapping q}), in the sense that OGD applied to non-convex problems can be approximated by OMD applied to convex functions under reparameterization, using which a sub-linear static regret bound is guaranteed.

More specifically, for a centralized problem, suppose that there is a non-linear mapping $q$, such that $f_t(\xb) = \Tilde{f}_t(q(\xb))$, where $\Tilde{f}_t(\cdot)$ is convex, and consider the OGD and OMD updates\\ 
{\bf OGD:}
\begin{equation}\label{Eq: Centralized OGD}
\resizebox{.47\textwidth}{!}{$
    \xb_{t+1} = \argmin_{\xb\in\mathcal{X}} \left\{\nabla f_t(\xb_t)^{\top}(\xb - \xb_t) + \frac{1}{2\eta}\norm{\xb - \xb_t}^2\right\},$}
\end{equation}
{\bf OMD:}
\begin{equation}\label{Eq: Centralized OMD}
\resizebox{.47\textwidth}{!}{$
    \ub_{t+1} = \argmin_{\ub\in\mathcal{X}^{\prime}} \left\{\nabla \Tilde{f}_t(\ub_t)^{\top}(\ub - \ub_t) + \frac{1}{\eta}\mathcal{D}_{\phi}(\ub,\ub_t)\right\},$}
\end{equation}
where $\ub_{t} = q(\xb_{t})$ and $\Xc'$ is the image of $\Xc$ under the mapping $q$. By quantifying the deviation $\norm{\ub_{t+1} - q(\xb_{t+1})}$ between OGD and OMD, \citet{ghai2022non} proved a sublinear static regret bound for the online non-convex optimization using OGD. 

In this work, we establish this equivalence for ``distributed" variants of OGD and OMD under the additional complexity that the constraint set is unknown, and it can only be approximated via \cref{alg:Constraint estimation}. We also point out that as opposed to the convex case, for the non-convex algorithm we apply max-consensus \cite{nejad2009max} after the exploration phase, where at each communication step, each agent keeps the estimate with the maximum norm among the estimates collected from its neighbors (including its own estimate). In \cite{nejad2009max}, it was shown that for a connected graph, it takes $D_G$ communication steps for all agents to reach consensus on the estimate. Then, all agents construct the feasible set based on that estimate, which changes line 4 of \cref{alg:D-safe-ogd-linear-constraint} to a distributed OGD on a common feasible set. For the technical analysis of the non-convex setting, we use the following assumptions.


\begin{assumption}\label{A: Geometric relation between phi and q}
    There exists a bijective mapping $q: \mathcal{X}^s \to \mathcal{X}^{s\prime}$ such that $[\nabla^2 \phi(\ub)]^{-1}=J_q(\xb)J_q(\xb)^{\top}$ where $\ub=q(\xb)$.
\end{assumption}
Examples that satisfy Assumption \ref{A: Geometric relation between phi and q} are provided in Section 3.1 of \cite{ghai2022non}. For example, if $\phi$ is the negative entropy (respectively, log barrier), we can use quadratic (respectively, exponential) reparameterization for $q$. \citet{amid2020reparameterizing} showed that in the continuous-time setup when Assumption \ref{A: Geometric relation between phi and q} holds, the mirror descent regularization induced by $\phi$ can be transformed back to the Euclidean regularization by $q^{-1}$, which implies the equivalence between OMD for convex functions and OGD for non-convex functions. Though in the discrete-time case, the exact equivalence does not hold, \citet{ghai2022non} showed that OGD for non-convex functions can still be approximated as OMD for convex functions, and the corresponding static regret bound is $O(T^{2/3})$.

\begin{assumption}\label{A: Properties of mapping q}
    Properties of the mapping $q(\cdot)$:
    \begin{itemize}
        \item There exists a mapping $q(\cdot)$ such that $f_{i,t}(\xb) = \Tilde{f}_{i,t}(q(\xb))$, where $\Tilde{f}_{i,t}(\cdot)$ is convex.
        \item $q(\cdot)$ is a $C^3$-diffeomorphism, and $J_{q}(\xb)$ is diagonal.
        \item {\color{black} For any $\mathcal{X}\subset \mathcal{X}^s$ which is compact and convex, $\Xc^{\prime}\triangleq q(\mathcal{X})$ is convex and compact.}
    \end{itemize}
\end{assumption}
We again refer the reader to Section 3.1 of \cite{ghai2022non} for examples related to \cref{A: Properties of mapping q}. In general, Assumptions \ref{A: Geometric relation between phi and q}-\ref{A: Properties of mapping q} provide sufficient conditions for the analysis.  
In this work, the focus is on analyzing the effect of (i) the constraint estimation as well as (ii) the distributed setup in non-convex online learning, and we also generalize the analysis of \cite{ghai2022non} to the dynamic regret.

\begin{assumption}\label{A: Several Lipschitz assumptions}
    Let $W>1$ be a constant. Assume that $q(\cdot)$ is $W$-Lipschitz, $\phi(\cdot)$ is $1$-strongly convex and smooth with its first and third derivatives upper bounded by $W$. The first and second derivatives of $q^{-1}(\cdot)$ are also bounded by $W$. For all $\ub \in \mathcal{X}^{s\prime}$, $\mathcal{D}_{\phi}(\ub, \cdot)$ is $W$-Lipschitz over $\mathcal{X}^{s\prime}$.
\end{assumption}

\begin{assumption}\label{A: Gradient and diameter bounds}
    $\norm{\nabla \Tilde{f}_{i,t}(\ub)}\leq G_F$ for all $\ub \in \mathcal{X}^{s\prime}$ and $\sup_{\ub,\zb\in\mathcal{X}^{s\prime}}\mathcal{D}_{\phi}(\ub,\zb)\leq D^{\prime}$.
\end{assumption}

\begin{assumption}\label{A: Jensen's inequality for the Bregman divergence}
    Let $\xb$ and $\{\yb_i\}^m_{i=1}$ be vectors in $\mathrm{R}^d$. The Bregman divergence satisfies the separate convexity in the following sense
    \begin{equation*}
        \mathcal{D}_{\phi}(\xb, \sum_{i}^m\alpha_i \yb_i)\leq \sum_{i}^m\alpha_i \mathcal{D}_{\phi}(\xb, \yb_i),
    \end{equation*}
    where $\alpha\in \Delta_m$ is on the $m-$dimensional simplex.
\end{assumption}
This assumption is satisfied by commonly used Bregman divergences, e.g., Euclidean distance and KL divergence. We refer interested readers to \cite{bauschke2001joint,shahrampour2018distributed} for more information.

In the following theorem, we prove that with high probability, the dynamic regret bound of D-Safe-OGD is $O(T^{2/3}{\color{black} \sqrt{\log T}} + T^{2/3}C^*_T)$.

\begin{theorem}\label{T: Dynamic regret bound (non-convex)}
    Suppose Assumptions \ref{A1}-\ref{A3} and \ref{A: Geometric relation between phi and q}-\ref{A: Jensen's inequality for the Bregman divergence} hold and $T=\Omega\left(\big(\frac{L^2}{m\gamma^2\sigma_{\zeta}^2}\log (\frac{d}{\delta})\big)^{3/2}\right)$. By running Algorithm \ref{alg:D-safe-ogd-linear-constraint} with $\gamma\leq \frac{\Delta^s}{LL_A}$, $\eta=\Theta(T^{-2/3})$, $T_0 = \Theta(T^{2/3})$ and $T_1 = \Theta(\log T^{\rho})$, we have with probability at least $(1-2\delta)$
    \begin{equation*}
        \xb_{i,t}\in \mathcal{X}^s,\;\forall i\in [m], t\in[T],~~~\text{and}
    \end{equation*}
    \begin{equation*}
        \mathbf{Reg}^d_{i,T}  = O(T^{2/3}{\color{black} \sqrt{\log T}} + T^{2/3}C^*_T),\;\forall i \in [m].  
    \end{equation*}
\end{theorem}
\begin{corollary}
Suppose that the comparator sequence is static over time, i.e., $\xb^*_t = \xb^*,\; \forall t\in [T]$. Then, the individual regret bound becomes $O(T^{2/3}{\color{black} \sqrt{\log T}})$, which recovers the static regret bound of \cite{ghai2022non} {\color{black} up to log factor}.
\end{corollary}

It is worth noting that though in the convex case, the estimation of unknown constraints exacerbates the regret bound (due to $O(T^{2/3})$ time spent on exploration), for the non-convex case, the resulting bound still matches the static regret of \cite{ghai2022non}, where there is no estimation error. In other words, there is no trade-off in this case as the static regret (without estimation error) is $O(T^{2/3})$ \cite{ghai2022non}.

\section*{Conclusion}
In this work, we considered safe distributed online optimization with an unknown set of linear constraints. The goal of the network is to ensure that the action sequence selected by each agent, which only has partial information about the global
function, is competitive to the centralized minimizers in hindsight without violating the safety constraints. To address this problem, we proposed D-Safe-OGD, where starting from a safe region, it allows all agents to perform exploration to estimate the unknown constraints in a distributed fashion. Then, distributed OGD is applied over the feasible sets formed by agents estimates. For convex functions, we proved a dynamic regret bound of $O(T^{2/3}{\color{black} \sqrt{\log T}} + T^{1/3}C^*_T)$, which recovers the static regret bound of \citet{chaudhary2021safe} for the centralized case (single agent). Then, we showed that for the non-convex setting, the dynamic regret is upper bounded by $O(T^{2/3}{\color{black} \sqrt{\log T}} + T^{2/3}C^*_T)$, which recovers the static regret bound of \citet{ghai2022non} for the centralized case (single agent) {\color{black} up to log factor}. Possible future directions include improving the regret using adaptive techniques and/or deriving comprehensive regret bounds in terms of other variation measures, such as $V_T$.



\bibliographystyle{IEEEtran}
\bibliography{references}

\begin{thebibliography}{10}
\providecommand{\url}[1]{#1}
\csname url@samestyle\endcsname
\providecommand{\newblock}{\relax}
\providecommand{\bibinfo}[2]{#2}
\providecommand{\BIBentrySTDinterwordspacing}{\spaceskip=0pt\relax}
\providecommand{\BIBentryALTinterwordstretchfactor}{4}
\providecommand{\BIBentryALTinterwordspacing}{\spaceskip=\fontdimen2\font plus
\BIBentryALTinterwordstretchfactor\fontdimen3\font minus
  \fontdimen4\font\relax}
\providecommand{\BIBforeignlanguage}[2]{{%
\expandafter\ifx\csname l@#1\endcsname\relax
\typeout{** WARNING: IEEEtran.bst: No hyphenation pattern has been}%
\typeout{** loaded for the language `#1'. Using the pattern for}%
\typeout{** the default language instead.}%
\else
\language=\csname l@#1\endcsname
\fi
#2}}
\providecommand{\BIBdecl}{\relax}
\BIBdecl

\bibitem{hazan2016introduction}
E.~Hazan, ``Introduction to online convex optimization,'' \emph{Foundations and
  Trends in Optimization}, vol.~2, no. 3-4, pp. 157--325, 2016.

\bibitem{zinkevich2003online}
M.~Zinkevich, ``Online convex programming and generalized infinitesimal
  gradient ascent,'' in \emph{Proceedings of the 20th International Conference
  on Machine Learning}, 2003, pp. 928--936.

\bibitem{hazan2007logarithmic}
E.~Hazan, A.~Agarwal, and S.~Kale, ``Logarithmic regret algorithms for online
  convex optimization,'' \emph{Machine Learning}, vol.~69, no. 2-3, pp.
  169--192, 2007.

\bibitem{hazan2017efficient}
E.~Hazan, K.~Singh, and C.~Zhang, ``Efficient regret minimization in non-convex
  games,'' in \emph{International Conference on Machine Learning (ICML)}.\hskip
  1em plus 0.5em minus 0.4em\relax PMLR, 2017, pp. 1433--1441.

\bibitem{ghai2022non}
U.~Ghai, Z.~Lu, and E.~Hazan, ``Non-convex online learning via algorithmic
  equivalence,'' \emph{arXiv preprint arXiv:2205.15235}, 2022.

\bibitem{besbes2015non}
O.~Besbes, Y.~Gur, and A.~Zeevi, ``Non-stationary stochastic optimization,''
  \emph{Operations research}, vol.~63, no.~5, pp. 1227--1244, 2015.

\bibitem{jadbabaie2015online}
A.~Jadbabaie, A.~Rakhlin, S.~Shahrampour, and K.~Sridharan, ``Online
  optimization: Competing with dynamic comparators,'' in \emph{Artificial
  Intelligence and Statistics}, 2015, pp. 398--406.

\bibitem{dixit2019online}
R.~Dixit, A.~S. Bedi, R.~Tripathi, and K.~Rajawat, ``Online learning with
  inexact proximal online gradient descent algorithms,'' \emph{IEEE
  Transactions on Signal Processing}, vol.~67, no.~5, pp. 1338--1352, 2019.

\bibitem{shahrampour2018distributed}
S.~Shahrampour and A.~Jadbabaie, ``Distributed online optimization in dynamic
  environments using mirror descent,'' \emph{IEEE Transactions on Automatic
  Control}, vol.~63, no.~3, pp. 714--725, 2018.

\bibitem{akbari2015distributed}
M.~Akbari, B.~Gharesifard, and T.~Linder, ``Distributed online convex
  optimization on time-varying directed graphs,'' \emph{IEEE Transactions on
  Control of Network Systems}, vol.~4, no.~3, pp. 417--428, 2015.

\bibitem{dobbe2020learning}
R.~Dobbe, P.~Hidalgo-Gonzalez, S.~Karagiannopoulos, R.~Henriquez-Auba, G.~Hug,
  D.~S. Callaway, and C.~J. Tomlin, ``Learning to control in power systems:
  Design and analysis guidelines for concrete safety problems,'' \emph{Electric
  Power Systems Research}, vol. 189, p. 106615, 2020.

\bibitem{luong2019applications}
N.~C. Luong, D.~T. Hoang, S.~Gong, D.~Niyato, P.~Wang, Y.-C. Liang, and D.~I.
  Kim, ``Applications of deep reinforcement learning in communications and
  networking: A survey,'' \emph{IEEE Communications Surveys \& Tutorials},
  vol.~21, no.~4, pp. 3133--3174, 2019.

\bibitem{aastrom2010feedback}
K.~J. {\AA}str{\"o}m and R.~M. Murray, \emph{Feedback systems}.\hskip 1em plus
  0.5em minus 0.4em\relax Princeton university press, 2010.

\bibitem{mokhtari2016online}
A.~Mokhtari, S.~Shahrampour, A.~Jadbabaie, and A.~Ribeiro, ``Online
  optimization in dynamic environments: Improved regret rates for strongly
  convex problems,'' in \emph{IEEE 55th Conference on Decision and Control
  (CDC)}, 2016, pp. 7195--7201.

\bibitem{chaudhary2021safe}
S.~Chaudhary and D.~Kalathil, ``Safe online convex optimization with unknown
  linear safety constraints,'' \emph{Proceedings of the AAAI Conference on
  Artificial Intelligence (AAAI)}, vol.~36, no.~6, pp. 6175--6182, 2022.

\bibitem{zhang2018adaptive}
L.~Zhang, S.~Lu, and Z.-H. Zhou, ``Adaptive online learning in dynamic
  environments,'' in \emph{Advances in neural information processing systems},
  2018, pp. 1323--1333.

\bibitem{zhang2017improved}
L.~Zhang, T.~Yang, J.~Yi, R.~Jin, and Z.-H. Zhou, ``Improved dynamic regret for
  non-degenerate functions,'' in \emph{Advances in Neural Information
  Processing Systems}, 2017, pp. 732--741.

\bibitem{hall2013dynamical}
E.~Hall and R.~Willett, ``Dynamical models and tracking regret in online convex
  programming,'' in \emph{International Conference on Machine Learning}.\hskip
  1em plus 0.5em minus 0.4em\relax PMLR, 2013, pp. 579--587.

\bibitem{rakhlin2013online}
A.~Rakhlin and K.~Sridharan, ``Online learning with predictable sequences,'' in
  \emph{Conference on Learning Theory}.\hskip 1em plus 0.5em minus 0.4em\relax
  PMLR, 2013, pp. 993--1019.

\bibitem{chang2021online}
T.-J. Chang and S.~Shahrampour, ``On online optimization: Dynamic regret
  analysis of strongly convex and smooth problems,'' \emph{Proceedings of the
  AAAI Conference on Artificial Intelligence}, vol.~35, no.~8, pp. 6966--6973,
  2021.

\bibitem{ravier2019prediction}
R.~J. Ravier, A.~R. Calderbank, and V.~Tarokh, ``Prediction in online convex
  optimization for parametrizable objective functions,'' in \emph{IEEE 58th
  Conference on Decision and Control (CDC)}, 2019, pp. 2455--2460.

\bibitem{ajalloeian2020inexact}
A.~Ajalloeian, A.~Simonetto, and E.~Dall’Anese, ``Inexact online
  proximal-gradient method for time-varying convex optimization,'' in
  \emph{American Control Conference (ACC)}, 2020, pp. 2850--2857.

\bibitem{yan2012distributed}
F.~Yan, S.~Sundaram, S.~Vishwanathan, and Y.~Qi, ``Distributed autonomous
  online learning: Regrets and intrinsic privacy-preserving properties,''
  \emph{IEEE Transactions on Knowledge and Data Engineering}, vol.~25, no.~11,
  pp. 2483--2493, 2012.

\bibitem{mateos2014distributed}
D.~Mateos-N{\'u}nez and J.~Cort{\'e}s, ``Distributed online convex optimization
  over jointly connected digraphs,'' \emph{IEEE Transactions on Network Science
  and Engineering}, vol.~1, no.~1, pp. 23--37, 2014.

\bibitem{hosseini2016online}
S.~Hosseini, A.~Chapman, and M.~Mesbahi, ``Online distributed convex
  optimization on dynamic networks,'' \emph{IEEE Transactions on Automatic
  Control}, vol.~61, no.~11, pp. 3545--3550, 2016.

\bibitem{dixit2019distributed}
R.~Dixit, A.~S. Bedi, K.~Rajawat, and A.~Koppel, ``Distributed online learning
  over time-varying graphs via proximal gradient descent,'' in \emph{2019 IEEE
  58th Conference on Decision and Control (CDC)}.\hskip 1em plus 0.5em minus
  0.4em\relax IEEE, 2019, pp. 2745--2751.

\bibitem{zhang2019distributed}
Y.~Zhang, R.~J. Ravier, M.~M. Zavlanos, and V.~Tarokh, ``A distributed online
  convex optimization algorithm with improved dynamic regret,'' in \emph{2019
  IEEE 58th Conference on Decision and Control (CDC)}.\hskip 1em plus 0.5em
  minus 0.4em\relax IEEE, 2019, pp. 2449--2454.

\bibitem{eshraghi2022improving}
N.~Eshraghi and B.~Liang, ``Improving dynamic regret in distributed online
  mirror descent using primal and dual information,'' in \emph{Learning for
  Dynamics and Control Conference}.\hskip 1em plus 0.5em minus 0.4em\relax
  PMLR, 2022, pp. 637--649.

\bibitem{lu2021online}
K.~Lu and L.~Wang, ``Online distributed optimization with nonconvex objective
  functions: Sublinearity of first-order optimality condition-based regret,''
  \emph{IEEE Transactions on Automatic Control}, vol.~67, no.~6, pp.
  3029--3035, 2021.

\bibitem{yu2020low}
H.~Yu and M.~J. Neely, ``{A Low Complexity Algorithm with $O(\sqrt{T})$ Regret
  and $O(1)$ Constraint Violations for Online Convex Optimization with Long
  Term Constraints},'' \emph{Journal of Machine Learning Research}, vol.~21,
  no.~1, pp. 1--24, 2020.

\bibitem{mahdavi2012trading}
M.~Mahdavi, R.~Jin, and T.~Yang, ``Trading regret for efficiency: online convex
  optimization with long term constraints,'' \emph{The Journal of Machine
  Learning Research}, vol.~13, no.~1, pp. 2503--2528, 2012.

\bibitem{jenatton2016adaptive}
R.~Jenatton, J.~Huang, and C.~Archambeau, ``Adaptive algorithms for online
  convex optimization with long-term constraints,'' in \emph{International
  Conference on Machine Learning}.\hskip 1em plus 0.5em minus 0.4em\relax PMLR,
  2016, pp. 402--411.

\bibitem{yuan2018online}
J.~Yuan and A.~Lamperski, ``Online convex optimization for cumulative
  constraints,'' \emph{Advances in Neural Information Processing Systems},
  vol.~31, 2018.

\bibitem{yi2021regret}
X.~Yi, X.~Li, T.~Yang, L.~Xie, T.~Chai, and K.~Johansson, ``Regret and
  cumulative constraint violation analysis for online convex optimization with
  long term constraints,'' in \emph{International Conference on Machine
  Learning}.\hskip 1em plus 0.5em minus 0.4em\relax PMLR, 2021, pp.
  11\,998--12\,008.

\bibitem{yuan2017adaptive}
D.~Yuan, D.~W. Ho, and G.-P. Jiang, ``An adaptive primal-dual subgradient
  algorithm for online distributed constrained optimization,'' \emph{IEEE
  transactions on cybernetics}, vol.~48, no.~11, pp. 3045--3055, 2017.

\bibitem{yuan2020distributed}
D.~Yuan, A.~Proutiere, and G.~Shi, ``Distributed online linear regressions,''
  \emph{IEEE Transactions on Information Theory}, vol.~67, no.~1, pp. 616--639,
  2020.

\bibitem{yuan2021distributed}
------, ``Distributed online optimization with long-term constraints,''
  \emph{IEEE Transactions on Automatic Control}, vol.~67, no.~3, pp.
  1089--1104, 2021.

\bibitem{yi2022regret}
X.~Yi, X.~Li, T.~Yang, L.~Xie, T.~Chai, and H.~Karl, ``Regret and cumulative
  constraint violation analysis for distributed online constrained convex
  optimization,'' \emph{IEEE Transactions on Automatic Control}, 2022.

\bibitem{lin2019competitive}
Q.~Lin, H.~Yi, J.~Pang, M.~Chen, A.~Wierman, M.~Honig, and Y.~Xiao,
  ``Competitive online optimization under inventory constraints,''
  \emph{Proceedings of the ACM on Measurement and Analysis of Computing
  Systems}, vol.~3, no.~1, pp. 1--28, 2019.

\bibitem{fan2020online}
S.~Fan, G.~He, X.~Zhou, and M.~Cui, ``Online optimization for networked
  distributed energy resources with time-coupling constraints,'' \emph{IEEE
  Transactions on Smart Grid}, vol.~12, no.~1, pp. 251--267, 2020.

\bibitem{hazan2012projection}
E.~Hazan and S.~Kale, ``Projection-free online learning,'' in
  \emph{International Coference on International Conference on Machine Learning
  (ICML)}, 2012, pp. 1843--1850.

\bibitem{sun2017safety}
W.~Sun, D.~Dey, and A.~Kapoor, ``Safety-aware algorithms for adversarial
  contextual bandit,'' in \emph{International Conference on Machine
  Learning}.\hskip 1em plus 0.5em minus 0.4em\relax PMLR, 2017, pp. 3280--3288.

\bibitem{chen2017online}
T.~Chen, Q.~Ling, and G.~B. Giannakis, ``An online convex optimization approach
  to proactive network resource allocation,'' \emph{IEEE Transactions on Signal
  Processing}, vol.~65, no.~24, pp. 6350--6364, 2017.

\bibitem{neely2017online}
M.~J. Neely and H.~Yu, ``Online convex optimization with time-varying
  constraints,'' \emph{arXiv preprint arXiv:1702.04783}, 2017.

\bibitem{yu2017online}
H.~Yu, M.~Neely, and X.~Wei, ``Online convex optimization with stochastic
  constraints,'' \emph{Advances in Neural Information Processing Systems},
  vol.~30, 2017.

\bibitem{cao2018online}
X.~Cao and K.~R. Liu, ``Online convex optimization with time-varying
  constraints and bandit feedback,'' \emph{IEEE Transactions on automatic
  control}, vol.~64, no.~7, pp. 2665--2680, 2018.

\bibitem{liu2022simultaneously}
Q.~Liu, W.~Wu, L.~Huang, and Z.~Fang, ``Simultaneously achieving sublinear
  regret and constraint violations for online convex optimization with
  time-varying constraints,'' \emph{ACM SIGMETRICS Performance Evaluation
  Review}, vol.~49, no.~3, pp. 4--5, 2022.

\bibitem{yi2020distributed}
X.~Yi, X.~Li, L.~Xie, and K.~H. Johansson, ``Distributed online convex
  optimization with time-varying coupled inequality constraints,'' \emph{IEEE
  Transactions on Signal Processing}, vol.~68, pp. 731--746, 2020.

\bibitem{yi2021distributed}
X.~Yi, X.~Li, T.~Yang, L.~Xie, T.~Chai, and K.~H. Johansson, ``Distributed
  bandit online convex optimization with time-varying coupled inequality
  constraints,'' \emph{IEEE Transactions on Automatic Control}, vol.~66,
  no.~10, pp. 4620--4635, 2020.

\bibitem{li2020online}
J.~Li, C.~Gu, Z.~Wu, and T.~Huang, ``Online learning algorithm for distributed
  convex optimization with time-varying coupled constraints and bandit
  feedback,'' \emph{IEEE transactions on cybernetics}, 2020.

\bibitem{sharma2021distributed}
P.~Sharma, P.~Khanduri, L.~Shen, D.~J. Bucci, and P.~K. Varshney, ``On
  distributed online convex optimization with sublinear dynamic regret and
  fit,'' in \emph{2021 55th Asilomar Conference on Signals, Systems, and
  Computers}.\hskip 1em plus 0.5em minus 0.4em\relax IEEE, 2021, pp.
  1013--1017.

\bibitem{li2022survey}
X.~Li, L.~Xie, and N.~Li, ``A survey of decentralized online learning,''
  \emph{arXiv preprint arXiv:2205.00473}, 2022.

\bibitem{amani2019linear}
S.~Amani, M.~Alizadeh, and C.~Thrampoulidis, ``Linear stochastic bandits under
  safety constraints,'' \emph{Advances in Neural Information Processing
  Systems}, vol.~32, 2019.

\bibitem{khezeli2020safe}
K.~Khezeli and E.~Bitar, ``Safe linear stochastic bandits,'' \emph{Proceedings
  of the AAAI Conference on Artificial Intelligence}, vol.~34, no.~06, pp.
  10\,202--10\,209, 2020.

\bibitem{usmanova2019safe}
I.~Usmanova, A.~Krause, and M.~Kamgarpour, ``Safe convex learning under
  uncertain constraints,'' in \emph{The 22nd International Conference on
  Artificial Intelligence and Statistics}.\hskip 1em plus 0.5em minus
  0.4em\relax PMLR, 2019, pp. 2106--2114.

\bibitem{fereydounian2020safe}
M.~Fereydounian, Z.~Shen, A.~Mokhtari, A.~Karbasi, and H.~Hassani, ``Safe
  learning under uncertain objectives and constraints,'' \emph{arXiv preprint
  arXiv:2006.13326}, 2020.

\bibitem{liu2008monte}
J.~S. Liu, \emph{Monte Carlo strategies in scientific computing}.\hskip 1em
  plus 0.5em minus 0.4em\relax Springer Science \& Business Media, 2008.

\bibitem{shi2015extra}
W.~Shi, Q.~Ling, G.~Wu, and W.~Yin, ``Extra: An exact first-order algorithm for
  decentralized consensus optimization,'' \emph{SIAM Journal on Optimization},
  vol.~25, no.~2, pp. 944--966, 2015.

\bibitem{nejad2009max}
B.~M. Nejad, S.~A. Attia, and J.~Raisch, ``Max-consensus in a max-plus
  algebraic setting: The case of fixed communication topologies,'' in
  \emph{2009 XXII international symposium on information, communication and
  automation technologies}.\hskip 1em plus 0.5em minus 0.4em\relax IEEE, 2009,
  pp. 1--7.

\bibitem{amid2020reparameterizing}
E.~Amid and M.~K. Warmuth, ``Reparameterizing mirror descent as gradient
  descent,'' \emph{Advances in Neural Information Processing Systems}, vol.~33,
  pp. 8430--8439, 2020.

\bibitem{bauschke2001joint}
H.~H. Bauschke and J.~M. Borwein, ``Joint and separate convexity of the bregman
  distance,'' in \emph{Studies in Computational Mathematics}.\hskip 1em plus
  0.5em minus 0.4em\relax Elsevier, 2001, vol.~8, pp. 23--36.

\bibitem{abbasi2011improved}
Y.~Abbasi-Yadkori, D.~P{\'a}l, and C.~Szepesv{\'a}ri, ``Improved algorithms for
  linear stochastic bandits,'' \emph{Advances in neural information processing
  systems}, vol.~24, 2011.

\bibitem{tropp2015introduction}
J.~A. Tropp \emph{et~al.}, ``An introduction to matrix concentration
  inequalities,'' \emph{Foundations and Trends{\textregistered} in Machine
  Learning}, vol.~8, no. 1-2, pp. 1--230, 2015.

\end{thebibliography}

\onecolumn

\newpage
\appendix

\section{Appendix}
In this section, we provide the proofs of our theoretical results. In Section \ref{sec:prelim}, we state the results we use in our analysis. Section \ref{sec:safe} includes the proof of estimation error bound in \cref{L: Bound for estimation error}. In Sections \ref{sec:conv} and \ref{sec:nonconv}, we provide the proofs for \cref{T: Dynamic regret bound} (convex case) and \cref{T: Dynamic regret bound (non-convex)} (non-convex case), respectively.

\subsection{Preliminaries:}\label{sec:prelim}
\begin{theorem}\label{T: Confidence set of OLS}(Theorem 2 in \cite{abbasi2011improved}).
Let $\{F_t\}_{t=0}^{\infty}$ be a filtration. Let $\{w_t\}_{t=1}^{\infty}$ be a real-valued stochastic process. Here $w_t$ is $F_t$-measurable and $w_t$ is conditionally $R$-sub Gaussian for some $R\geq 0$. Let $\{\xb_t\}_{t=1}^{\infty}$ be an $\mathrm{R}^d$-valued stochastic process such that $\xb_t$ is $F_{t-1}$-measurable. Let $\Vb_T\triangleq \sum_{t=1}^T \xb_t\xb_t^{\top} + \lambda\Ib$ where $\lambda>0$. Define $y_t$ as $\ab^{\top}\xb_t + w_t$, then $\hat{\ab}_T = \Vb_T^{-1} \sum_{t=1}^T y_t \xb_t$ is the $l_2$-regularized least squares estimate of $\ab$. Assume $\norm{\ab}\leq L_A$ and $\norm{\xb_t}\leq L,\; \forall t$. Then for any $\delta > 0$, with probability $(1-\delta)$, the true parameter $\ab$ lies in the following set:
\begin{equation*}
    \left\{ \ab\in\mathrm{R}^d: \norm{\ab - \hat{\ab}_T}_{\Vb_T} \leq R\sqrt{d\log \big(\frac{1+TL^2/\lambda}{\delta}\big)} + \sqrt{\lambda}L_A \right\},
\end{equation*}
for all $T\geq 1$.
\end{theorem}

\begin{theorem}\label{T: Matrix concentration inequality}(Theorem 5.1.1 in \cite{tropp2015introduction}).
Consider a finite sequence $\{\Xb_t\}$ of independent, random and positive semi-definite matrices $\in \mathrm{R}^{d\times d}$. Assume that $\lambda_{\max}(\Xb_t)\leq L, \;\forall t$. Define $\Yb\triangleq \sum_{t}\Xb_t$ and denote $\lambda_{\min}(\mathbb{E}[\Yb])$ as $\mu$. Then we have
\begin{equation*}
    \mathbb{P}(\lambda_{\min}(\Yb)\leq \epsilon\mu) \leq d\exp\big(-(1-\epsilon)^2\frac{\mu}{2L}\big),\;\text{for any }\epsilon \in (0,1).
\end{equation*}
\end{theorem}

\noindent
Definition of a {\bf shrunk} polytope:
\begin{align}
\label{eq:inner-polytope}
    \mathcal{X}^{s}_{\textnormal{in}} &= \{ \xb \in \mathbb{R}^{d}:  \ab_k^\top \xb + \tau_{\textnormal{in}} \leq b_k, \forall k \in [n]\},\; \textnormal{for some } \tau_{\textnormal{in}} > 0.
\end{align}

\begin{lemma}[Lemma 1 in \cite{fereydounian2020safe}]
\label{lem:fereydounian2020safe-lemma}
Consider a positive constant $\tau_{\textnormal{in}}$ such that $\mathcal{X}^{s}_{\textnormal{in}}$ is non-empty. Then, for any $\xb \in \mathcal{X}^{s}$,
\begin{align}
    \|\Pi_{\mathcal{X}^{s}_{\textnormal{in}}}(\xb) - \xb\| \leq \frac{\sqrt{d} \tau_{\textnormal{in}}}{C(\Ab, \bb)},
\end{align}
where $C(\Ab, \bb)$ is a positive constant that depends only on the matrix $\Ab$ and the vector $\bb$. 
\end{lemma}
\noindent

\subsection{Safe Distributed Set Estimation}\label{sec:safe}
\begin{proof}[Proof of Lemma \ref{L: Bound for estimation error}]
Let $\Vb_{T_0}\triangleq\sum_{i=1}^m\sum_{t=1}^{T_0}\xb_{i,t}\xb_{i,t}^{\top}$ and $\Vb = \Vb_{T_0} + \lambda\Ib$. Let $\widehat{\Ab}$ be the solution of $\argmin_{\Ab}\sum_{i=1}^m l_i(\Ab).$ 
Let $\hat{\ab}_k$ and $\ab_k$ be the $k$-th rows of $\widehat{\Ab}$ and $\Ab$, respectively. Based on Theorem \ref{T: Confidence set of OLS}, we have with probability at least $(1-\delta)$,
\begin{equation}\label{Eq1: closeness of agents' estimates}
    \norm{\hat{\ab}_k - \ab_k}_{\Vb} \leq R\sqrt{d\log \big(\frac{1+mT_0L^2/\lambda}{\delta/n}\big)} + \sqrt{\lambda}L_A,\; \forall k\in[n].
\end{equation}
Knowing that $\forall i \in [m],\;\forall t\in[T_0]$, $\xb_{i,t}= (1-\gamma)\xb^s + \gamma\zeta_{i,t}$, we have $\lambda_{\max}(\xb_{i,t}\xb_{i,t}^{\top})\leq L^2$ and $\mathbb{E}[\xb_{i,t}\xb_{i,t}^{\top}] = (1-\gamma)^2\xb^s\xb^{s\top} + \gamma^2 \sigma_{\zeta}^2\Ib\succeq \gamma^2\sigma_{\zeta}^2\Ib$. Therefore, we have 
\begin{equation}\label{Eq2: closeness of agents' estimates}
    \lambda_{\min}(\mathbb{E}[\Vb_{T_0}]) = \lambda_{\min}(\sum_{i=1}^m\sum_{t=1}^{T_0}\mathbb{E}[\xb_{i,t}\xb_{i,t}^{\top}])\geq mT_0\gamma^2\sigma_{\zeta}^2.
\end{equation}
Based on Equation \eqref{Eq2: closeness of agents' estimates} and Theorem \ref{T: Matrix concentration inequality}, we have 
\begin{equation}\label{Eq3: closeness of agents' estimates}
    \mathbb{P}(\lambda_{\min}(\Vb_{T_0})\leq \epsilon mT_0\gamma^2\sigma_{\zeta}^2) \leq d\exp\big(-(1-\epsilon)^2\frac{mT_0\gamma^2\sigma_{\zeta}^2}{2L^2}\big).
\end{equation}
By setting $\epsilon=\frac{1}{2}$ and $T_0\geq \frac{8L^2}{m\gamma^2\sigma_{\zeta}^2}\log (\frac{d}{\delta})$, from Equation \eqref{Eq3: closeness of agents' estimates}, we have
\begin{equation}\label{Eq4: closeness of agents' estimates}
    \mathbb{P}\big(\lambda_{\min}(\Vb)\geq \frac{1}{2}mT_0\gamma^2\sigma_{\zeta}^2\big)\geq \mathbb{P}\big(\lambda_{\min}(\Vb_{T_0})\geq \frac{1}{2}mT_0\gamma^2\sigma_{\zeta}^2\big)\geq (1-\delta).
\end{equation}
Combining Equations \eqref{Eq1: closeness of agents' estimates} and \eqref{Eq4: closeness of agents' estimates}, we have with probability at least $(1-2\delta)$, 
\begin{equation}\label{Eq5: closeness of agents' estimates}
    \norm{\hat{\ab}_k - \ab_k} \leq \frac{R\sqrt{d\log \big(\frac{1+mT_0L^2/\lambda}{\delta/n}\big)} + \sqrt{\lambda}L_A}{\sqrt{\frac{1}{2}m\gamma^2\sigma_{\zeta}^2T_0}},\; \forall k\in[n].
\end{equation}

Let agent $i$'s local estimate of A at time $t \in [T_0+1,T_0+T_1]$ returned by the {\bf EXTRA} algorithm \cite{shi2015extra} be denoted by $\widehat{\Ab}_{i,t}$. Next we upper bound the distance between $\widehat{\Ab} = \argmin_{\Ab}\sum_{i=1}^m l_i(\Ab)$ and $\widehat{\Ab}_{i,t}$ based on Theorem 3.7 in \cite{shi2015extra} as follows. There exists $0<\tau < 1$ such that 
\begin{equation}\label{Eq6: closeness of agents' estimates}
    \norm{\hat{\ab}^i_{k,t}-\hat{\ab}_k} \leq \nu\tau^{(t-T_0)},\; \forall i \in [m], k \in [n], t \in [T_0+1,\ldots,T_0+T_1]
\end{equation} 
where $\nu$ is a constant. Based on \eqref{Eq5: closeness of agents' estimates}, \eqref{Eq6: closeness of agents' estimates} and our choice of $T_1$ ($T_1=(-\log\tau)^{-1}\log(\nu T^{\rho})$), for $k\in[n]$, $t\in[T_0+1,\ldots,T_0+T_1]$ and $i,j\in[m]$, we have
\begin{equation}\label{Eq7: closeness of agents' estimates}
    \norm{\hat{\ab}^i_{k,t}-\ab_k} \leq \norm{\hat{\ab}^i_{k,t}-\hat{\ab}_k} + \norm{\hat{\ab}_k - \ab_k}
    \leq \frac{1}{T^{\rho}} + \frac{R\sqrt{d\log \big(\frac{1+mT_0L^2/\lambda}{\delta/n}\big)} + \sqrt{\lambda}L_A}{\sqrt{\frac{1}{2}m\gamma^2\sigma_{\zeta}^2T_0}},
\end{equation}
and
\begin{equation}\label{Eq8: closeness of agents' estimates}
    \norm{\hat{\ab}^i_{k,t}-\hat{\ab}^j_{k,t}} \leq \norm{\hat{\ab}^i_{k,t}-\hat{\ab}_k} + \norm{\hat{\ab}_k - \hat{\ab}^j_{k,t}}
    \leq \frac{2}{T^{\rho}}.
\end{equation}
\end{proof}



\begin{lemma}\label{L: Similarity property of different agents' feasible sets}
    Define $$\mathcal{B}_r \triangleq \frac{1}{T^{\rho}} + \frac{R\sqrt{d\log \big(\frac{1+mT_0L^2/\lambda}{\delta/n}\big)} + \sqrt{\lambda}L_A}{\sqrt{\frac{1}{2}m\gamma^2\sigma_{\zeta}^2T_0}}.$$ For each agent $i$, construct $\widehat{\mathcal{X}}^s_i$ based on Equation \eqref{Eq: Agent's feasible set} with $\mathcal{C}_{i,k}$ which follows Equation \eqref{Eq: The ball set w.r.t. an estimate} with $\mathcal{B}_r$. By running Algorithm \ref{alg:Constraint estimation} with user-specified $T_0$ and $T_1=O(\log T^{\rho})$, we have, with probability at least $(1-2\delta)$,
    \begin{equation}
        \norm{\Pi_{\widehat{\mathcal{X}}^s_j}(\xb_i) - \xb_i}\leq \frac{2\sqrt{d}L\mathcal{B}_r}{C(\Ab, \bb)},\;\forall i,j \in [m],
    \end{equation}
    where $\xb_i\in \widehat{\mathcal{X}}^s_i$.
\end{lemma}

\begin{proof}[Proof of Lemma \ref{L: Similarity property of different agents' feasible sets}]
First we show that there exists a mutual shrunk polytope (see Definition \ref{eq:inner-polytope}) subset $\mathcal{X}^{s}_{\textnormal{in}}$ ($\tau_{\textnormal{in}} = 2\mathcal{B}_rL$) for $\widehat{\mathcal{X}}^s_i,\:\forall i \in [m]$. Based on Lemma \ref{L: Bound for estimation error}, with probability at least ${1-2\delta}$, we have for any $\xb\in \mathcal{X}^s_{\textnormal{in}}$,
\begin{equation}\label{Eq1: Similarity property of different agents' feasible sets}
\begin{split}
    \hat{\ab}^{i\top}_k \xb + \mathcal{B}_r\norm{\xb} &= \ab_k^{\top}\xb + (\hat{\ab}^{i}_k - \ab_k)^{\top}\xb + \mathcal{B}_r\norm{\xb}\\
    &\leq \ab_k^{\top}\xb + \norm{\hat{\ab}^{i}_k - \ab_k}\norm{\xb} + \mathcal{B}_r\norm{\xb}\\
    &\leq \ab_k^{\top}\xb + 2\mathcal{B}_r\norm{\xb}\leq \ab_k^{\top}\xb + 2\mathcal{B}_r L\leq b_k,\; \forall k\in [n]\textnormal{ and }\forall i \in [m],
\end{split}
\end{equation}
which implies that $\mathcal{X}^{s}_{\textnormal{in}}\subset \widehat{\mathcal{X}}^s_i,\:\forall i$. 
Based on Lemma \ref{lem:fereydounian2020safe-lemma} and the fact that $\mathcal{X}^{s}_{\textnormal{in}}\subset \widehat{\mathcal{X}}^s_i \textnormal{ and }\widehat{\mathcal{X}}^s_i\subset \mathcal{X}^s,\:\forall i$, we have that $\forall i,j \in [m]$ and any $\xb_i\in \widehat{\mathcal{X}}^s_i$
\begin{equation}\label{Eq2: Similarity property of different agents' feasible sets}
\begin{split}
    \norm{\Pi_{\widehat{\mathcal{X}}^s_j}(\xb_i) - \xb_i} &\leq \norm{\Pi_{\mathcal{X}^{s}_{\textnormal{in}}}(\xb_i) - \xb_i}\\
    &\leq \max_{\xb_i\in \widehat{\mathcal{X}}^s_i}\norm{\Pi_{\mathcal{X}^{s}_{\textnormal{in}}}(\xb_i) - \xb_i}\\
    &\leq \max_{\xb\in \mathcal{X}^s}\norm{\Pi_{\mathcal{X}^{s}_{\textnormal{in}}}(\xb) - \xb}\leq \frac{2\sqrt{d}\mathcal{B}_rL}{C(\Ab, \bb)},
\end{split}    
\end{equation}
where the first and third inequalities are due to the facts that $\mathcal{X}^{s}_{\textnormal{in}}\subset \widehat{\mathcal{X}}^s_i \textnormal{ and }\widehat{\mathcal{X}}^s_i\subset \mathcal{X}^s$, respectively.
\end{proof}

\subsection{Convex Part}\label{sec:conv}
\begin{lemma}\label{L: consecutive distance}
Let Algorithm \ref{alg:D-safe-ogd-linear-constraint} run with step size $\eta>0$ and define $\xb_t\triangleq\frac{1}{m}\sum_{i=1}^m\xb_{i,t}$ and $\yb_t\triangleq\frac{1}{m}\sum_{i=1}^m\yb_{i,t}$. Under Assumptions \ref{A1} to \ref{A4}, we have that $\forall i \in [m]$
\begin{equation*}
\begin{split}
    \norm{\xb_t - \xb_{i,t}}\leq (\frac{2\sqrt{d}L\mathcal{B}_r}{C(\Ab, \bb)} +2\eta G)\frac{\sqrt{m}\beta^{}}{1-\beta^{}} = O\big(\eta + \mathcal{B}_r\big) \text{ and } \norm{\yb_t - \yb_{i,t}}(\frac{2\sqrt{d}L\mathcal{B}_r}{C(\Ab, \bb)} +2\eta G)\frac{\sqrt{m}}{1-\beta},
\end{split}
\end{equation*}
where $\mathcal{B}_r$ follows the definition in Lemma \ref{L: Similarity property of different agents' feasible sets}.
\end{lemma}

\begin{proof}
    For the sake of simplicity, we define the following matrices
\begin{equation*}
\begin{split}
    \Xb_t &\triangleq [\xb_{1,t},\ldots,\xb_{m,t}],\;\Yb_t \triangleq [\yb_{1,t},\ldots,\yb_{m,t}],\;\Gb_t\triangleq [\nabla f_{1,t}(\xb_{1,t}),\ldots,\nabla f_{m,t}(\xb_{m,t})],\text{ and }\Rb_t \triangleq [r_{1,t},\ldots,r_{m,t}],
\end{split}
\end{equation*}
where $r_{i,t-1} = \yb_{i,t-1} - \big(\xb_{i,t-1}-\eta\nabla f_{i,t-1}(\xb_{i,t-1})\big)$. Then the update can be expressed as $\Xb_{t} = \Yb_{t-1}\Pb^{} = \big(\Xb_{t-1} - \eta \Gb_{t-1} - \Rb_{t-1}\big)\Pb^{}$.

Expanding the update recursively, we have
\begin{equation}\label{Eq1: consecutive distance}
    \Xb_t = \Xb_{T_s}\Pb^{(t-T_s)} - \eta\sum_{l=1}^{t-T_s}\Gb_{t-l}\Pb^{ l} - \sum_{l=1}^{t-T_s}\Rb_{t-l}\Pb^{ l}.
\end{equation}
Since $\Pb$ is doubly stochastic, we have $\Pb^k\mathbf{1}=\mathbf{1}$ for all $k\geq 1$. Based on the geometric mixing bound of $\Pb$ and Equations \eqref{Eq1: consecutive distance}, we get
\begin{equation*}
\begin{split}
    &\norm{\xb_t - \xb_{i,t}}=\norm{\Xb_t(\frac{1}{m}\mathbf{1}-\eb_i)}\\
    \leq &\norm{\xb_{T_s}-\Xb_{T_s}[\Pb^{(t-T_s)}]_{:,i}} + \eta\sum_{l=1}^{t-T_s}\norm{\Gb_{t-l}(\frac{1}{m}\mathbf{1}-[\Pb^{ l}]_{:,i})}
    +\sum_{l=1}^{t-T_s}\norm{\Rb_{t-l}(\frac{1}{m}\mathbf{1}-[\Pb^{ l}]_{:,i})}\\
    \leq &\sum_{l=1}^{t-T_s}(\eta G)\sqrt{m}\beta^{ l} + \sum_{l=1}^{t-T_s}(\frac{2\sqrt{d}L\mathcal{B}_r}{C(\Ab, \bb)} + \eta G)\sqrt{m}\beta^{ l}\\
    \leq &(\frac{2\sqrt{d}L\mathcal{B}_r}{C(\Ab, \bb)} +2\eta G)\frac{\sqrt{m}\beta^{}}{1-\beta^{}},
\end{split}    
\end{equation*}
where $\norm{\xb_{T_s}-\Xb_{T_s}[\Pb^{t-T_s}]_{:,i}}=0$ by the identical initialization (of all agents with the same action at $T_s$) and based on Lemma \ref{L: Similarity property of different agents' feasible sets},
\begin{equation*}
\begin{split}
    \norm{r_{i,t}} &= \norm{\yb_{i,t} - \big(\xb_{i,t}-\eta\nabla f_{i,t}(\xb_{i,t})\big)}\\
    &\leq \norm{\sum_j [\Pb^{}]_{ji}\Pi_{\widehat{\mathcal{X}}^s_i}[\yb_{j,t-1}] - \big(\sum_j [\Pb^{}]_{ji}\yb_{j,t-1}-\eta\nabla f_{i,t}(\xb_{i,t})\big)}\\
    &\leq \frac{2\sqrt{d}L\mathcal{B}_r}{C(\Ab, \bb)} + \eta G.
\end{split}
\end{equation*}
Following the same manner, it can be shown that 
\begin{equation*}
    \norm{\yb_t - \yb_{i,t}} \leq (\frac{2\sqrt{d}L\mathcal{B}_r}{C(\Ab, \bb)} +2\eta G)\frac{\sqrt{m}}{1-\beta}.
\end{equation*}

\end{proof}

\begin{proof}[Proof of Theorem \ref{T: Dynamic regret bound}]
First, we decompose the individual regret of agent $j$ into three terms:
\begin{equation}\label{Eq1: Dynamic regret bound}
    \sum_{t}\sum_i f_{i,t}(\xb_{j,t}) - \sum_t f_t(\xb^*_t) = \underbrace{\sum_{t=1}^{T_s-1}\sum_i f_{i,t}(\xb_{j,t}) - f_{i,t}(\xb^*_t)}_{\text{Term I}} + \underbrace{\sum_{t=T_s}^{T}\sum_i f_{i,t}(\xb_{j,t}) - f_{i,t}(\Tilde{\xb}^*_t)}_{\text{Term II}} + \underbrace{\sum_{t=T_s}^{T} f_{t}(\Tilde{\xb}^*_t) - f_{t}(\xb^*_t)}_{\text{Term III}},
\end{equation}
where $\Tilde{\xb}^*_t$ is the projection of $\xb^*_t$ on $\mathcal{X}^s_{\textnormal{in}}$, which is a mutual subset of $\{\widehat{\mathcal{X}}^s_i\}_{i \in [m]}$ with $\tau_{\textnormal{in}} = 2\mathcal{B}_rL$ based on Equation \eqref{Eq1: Similarity property of different agents' feasible sets} in Lemma \ref{L: Similarity property of different agents' feasible sets}.

{\bf The upper bound of Term I}:\\
Here we note that by choosing $\gamma \leq \frac{\Delta^s}{LL_A}$, we have $\forall i \in [m]$ and $t\in [1,\ldots,T_0+T_1]$
\begin{equation}\label{Eq2: Dynamic regret bound}
    \ab_k^{\top}\xb_{i,t} = \ab_k^{\top}\left((1-\gamma)\xb^s +\gamma\zeta_{i,t}\right) \leq (1-\gamma)b^s_k + \Delta^s \leq (1-\gamma)b^s_k + (b_k - b^s_k) < b_k,
\end{equation}
which implies the safeness of the action.

Based on the Lipschitz property of the function sequence, we have
\begin{equation}\label{Eq3: Dynamic regret bound}
    \sum_{t=1}^{T_s-1}\sum_i f_{i,t}(\xb_{j,t}) - f_{i,t}(\xb^*_t) \leq \sum_{t=1}^{T_s-1}\sum_i G\norm{\xb_{j,t} - \xb^*_t} \leq 2GLm(T_0+T_1).
\end{equation}
{\bf The upper bound of Term II}:\\
Based on the update rule, $\forall i\in [m]$ and $t\in[T_s,\ldots, T]$ we have
\begin{equation}\label{Eq4: Dynamic regret bound}
\begin{split}
    &f_{i,t}(\xb_{i,t}) - f_{i,t}(\Tilde{\xb}^*_t)\\
    \leq &\nabla f_{i,t}(\xb_{i,t})^{\top}(\xb_{i,t} - \Tilde{\xb}^*_t)\\
    = &\frac{1}{\eta}\left[\frac{1}{2}\eta^2\norm{\nabla f_{i,t}(\xb_{i,t})}^2 + \frac{1}{2}\norm{\xb_{i,t}-\Tilde{\xb}^*_t}^2 - \frac{1}{2}\norm{\xb_{i,t}-\Tilde{\xb}^*_t - \eta \nabla f_{i,t}(\xb_{i,t})}^2\right]\\
    \leq &\frac{1}{\eta}\left[\frac{1}{2}\eta^2\norm{\nabla f_{i,t}(\xb_{i,t})}^2 + \frac{1}{2}\norm{\xb_{i,t}-\Tilde{\xb}^*_t}^2 - \frac{1}{2}\norm{\yb_{i,t} -\Tilde{\xb}^*_t}^2\right]\\
    = &\frac{1}{\eta}\left[\frac{1}{2}\eta^2\norm{\nabla f_{i,t}(\xb_{i,t})}^2 + \frac{1}{2}\norm{\sum_j [\Pb]_{ji}\yb_{j,t-1}-\Tilde{\xb}^*_t}^2 - \frac{1}{2}\norm{\yb_{i,t} -\Tilde{\xb}^*_t}^2\right]\\
    \leq &\frac{1}{\eta}\left[\frac{1}{2}\eta^2\norm{\nabla f_{i,t}(\xb_{i,t})}^2 + \frac{1}{2}\sum_j [\Pb]_{ji}\norm{\yb_{j,t-1}-\Tilde{\xb}^*_t}^2 - \frac{1}{2}\norm{\yb_{i,t} -\Tilde{\xb}^*_t}^2\right],
\end{split}    
\end{equation}
where the second inequality is due to the projection such that $\norm{\yb_{i,t} - \Tilde{\xb}^*_t}\leq \norm{\xb_{i,t} - \eta \nabla f_{i,t}(\xb_{i,t}) - \Tilde{\xb}^*_t}$ and the third inequality is due to the convexity of the square function.

Based on Equation \eqref{Eq4: Dynamic regret bound} and Lemma \ref{L: consecutive distance}, we have
\begin{equation}\label{Eq5: Dynamic regret bound}
\begin{split}
    &f_{i,t}(\xb_{j,t}) - f_{i,t}(\Tilde{\xb}^*_t)\\
    =& f_{i,t}(\xb_{j,t}) - f_{i,t}(\xb_{i,t}) + f_{i,t}(\xb_{i,t}) - f_{i,t}(\Tilde{\xb}^*_t)\\
    \leq& G\norm{\xb_{j,t} - \xb_{i,t}} + f_{i,t}(\xb_{i,t}) - f_{i,t}(\Tilde{\xb}^*_t)\\
    \leq & 2G\big((\frac{2\sqrt{d}L\mathcal{B}_r}{C(\Ab, \bb)} + 2\eta G)\frac{\sqrt{m}\beta}{1-\beta}\big) + \frac{1}{2}\eta\norm{\nabla f_{i,t}(\xb_{i,t})}^2 + \frac{1}{2\eta}\sum_j [\Pb]_{ji}\norm{\yb_{j,t-1}-\Tilde{\xb}^*_t}^2 - \frac{1}{2\eta}\norm{\yb_{i,t} - \Tilde{\xb}^*_t}^2.
\end{split}    
\end{equation}
Summing Equation \eqref{Eq5: Dynamic regret bound} over $i$, we get
\begin{equation}\label{Eq6: Dynamic regret bound}
\begin{split}
    &\sum_i \left(f_{i,t}(\xb_{j,t}) - f_{i,t}(\Tilde{\xb}^*_t)\right)\\
    \leq & 2mG\big((\frac{2\sqrt{d}L\mathcal{B}_r}{C(\Ab, \bb)} + 2\eta G)\frac{\sqrt{m}\beta}{1-\beta}\big) + \frac{\eta}{2}\sum_i \norm{\nabla f_{i,t}(\xb_{i,t})}^2  + \frac{1}{2\eta}\sum_j\norm{\yb_{j,t-1}-\Tilde{\xb}^*_t}^2 - \frac{1}{2\eta}\sum_i\norm{\yb_{i,t} - \Tilde{\xb}^*_t}^2\\
    = & 2mG\big((\frac{2\sqrt{d}L\mathcal{B}_r}{C(\Ab, \bb)} + 2\eta G)\frac{\sqrt{m}\beta}{1-\beta}\big) + \frac{\eta}{2}\sum_i \norm{\nabla f_{i,t}(\xb_{i,t})}^2 + \frac{1}{2\eta}\sum_i \left(\norm{\yb_{i,t-1}}^2 - \norm{\yb_{i,t}}^2 + 2(\yb_{i,t} - \yb_{i,t-1})^{\top}\Tilde{\xb}^*_t\right).
\end{split}    
\end{equation}
Summing Equation \eqref{Eq6: Dynamic regret bound} over $t\in[T_s,\ldots,T]$, we have
\begin{equation}\label{Eq7: Dynamic regret bound}
\begin{split}
    &\sum_{t=T_s}^{T}\sum_i \left(f_{i,t}(\xb_{j,t}) - f_{i,t}(\Tilde{\xb}^*_t)\right)\\
    \leq & \frac{\eta}{2}\sum_{t=T_s}^{T}\sum_i \norm{\nabla f_{i,t}(\xb_{i,t})}^2
    + \frac{1}{2\eta}\sum_i \norm{\yb_{i,T_s-1}}^2 + \frac{1}{\eta}\big(\sum_i \yb_{i,T}^{\top}\Tilde{\xb}_{T}^* - \sum_i \yb_{i,T_s-1}^{\top}\Tilde{\xb}^*_{T_s}\big) + \frac{1}{\eta}\sum_{t=T_s}^{T-1}\sum_i(\Tilde{\xb}^*_{t}-\Tilde{\xb}^*_{t+1})^{\top}\yb_{i,t}\\
    + &2TmG\big((\frac{2\sqrt{d}L\mathcal{B}_r}{C(\Ab, \bb)} + 2\eta G)\frac{\sqrt{m}\beta}{1-\beta}\big).
\end{split}
\end{equation}

{\bf The upper bound of Term III}:\\

Based on Lemma \ref{lem:fereydounian2020safe-lemma}, we have for any $\xb^*_t\in \mathcal{X}^s$ and its projection to $\mathcal{X}^s_{\textnormal{in}}$: $\Tilde{\xb}^*_t$
\begin{equation}\label{Eq8: Dynamic regret bound}
\begin{split}
    \sum_{t=T_s}^T\sum_i \left(f_{i,t}(\Tilde{\xb}^*_t) - f_{i,t}(\xb^*_t)\right) \leq \sum_{t=T_s}^T\sum_i G \norm{\Tilde{\xb}^*_t - \xb^*_t} \leq mT G \frac{2\sqrt{d}L\mathcal{B}_r}{C(\Ab, \bb)}.
\end{split}
\end{equation}

Substituting Equations \eqref{Eq3: Dynamic regret bound}, \eqref{Eq7: Dynamic regret bound} and \eqref{Eq8: Dynamic regret bound} into Equation \eqref{Eq1: Dynamic regret bound}, we get
\begin{equation}\label{Eq9: Dynamic regret bound}
   \begin{split}
        &\sum_t\sum_i \left(f_{i,t}(\xb_{j,t}) - f_{i,t}(\xb^*_t)\right)\\
        \leq & O(T_0 + T_1) + \frac{\eta mT G^2}{2} + \frac{1}{2\eta}\sum_i \norm{\yb_{i,T_s-1}}^2 + \frac{1}{\eta}\big(\sum_i \yb_{i,T}^{\top}\Tilde{\xb}_{T}^* - \sum_i \yb_{i,T_s-1}^{\top}\Tilde{\xb}^*_{T_s}\big) + \frac{1}{\eta}\sum_{t=T_s}^{T-1}\sum_i(\Tilde{\xb}^*_{t}-\Tilde{\xb}^*_{t+1})^{\top}\yb_{i,t}\\
        + &2TmG\big((\frac{2\sqrt{d}L\mathcal{B}_r}{C(\Ab, \bb)} + 2\eta G)\frac{\sqrt{m}\beta}{1-\beta}\big) + mT G \frac{2\sqrt{d}L\mathcal{B}_r}{C(\Ab, \bb)},
    \end{split}
\end{equation}
which is {\color{black} $O(T_0 + T_1 + \frac{1}{\eta} + \frac{1}{\eta}C^*_T + \frac{\beta T\sqrt{\log T_0}}{(1-\beta) \sqrt{T_0}} + \frac{\beta\eta T}{(1-\beta)})$} and the final regret bound is derived by substituting the choices of $\eta$ and $T_0$.
\end{proof}

\subsection{Non-convex Part}\label{sec:nonconv}
\begin{lemma}\label{L: consecutive distance2}
Let Algorithm \ref{alg:D-safe-ogd-linear-constraint} run with step size $\eta>0$ and the max-consensus step after the exploration phase. Under Assumptions \ref{A1} to \ref{A4}, we have that $\forall i \in [m]$
\begin{equation*}
\begin{split}
    \norm{\xb_t - \xb_{i,t}}\leq 2\eta G\frac{\sqrt{m}\beta^{}}{1-\beta^{}},
\end{split}
\end{equation*}
where $\xb_t\triangleq\frac{1}{m}\sum_{i=1}^m\xb_{i,t}$.
\end{lemma}

\begin{proof}
Similar proof as Lemma \ref{L: consecutive distance}. The only difference is that since the estimated feasible set is common over agents based on the max-consensus step, $\norm{r_{i,t}}$ is upper-bounded as follows:
\begin{equation*}
\begin{split}
    \norm{r_{i,t}} &= \norm{\yb_{i,t} - \big(\xb_{i,t}-\eta\nabla f_{i,t}(\xb_{i,t})\big)}\\
    &\leq \norm{\sum_j [\Pb]_{ji}\yb_{j,t-1} - \big(\sum_j [\Pb]_{ji}\yb_{j,t-1}-\eta\nabla f_{i,t}(\xb_{i,t})\big)}\\
    &\leq \eta G.
\end{split}
\end{equation*}
By following the similar proof in Lemma \ref{L: consecutive distance} based on the expression above, the result is approved.

\end{proof}

\begin{lemma}[Lemma 4 in \cite{ghai2022non}]\label{L: Update deviation between OGD and OMD}
    Suppose Assumptions \ref{A: Geometric relation between phi and q} to \ref{A: Gradient and diameter bounds} hold and $\ub_t = q(\xb_t)$, then $\norm{q(\xb_{t+1}) - \ub_{t+1}}=O(W^4G_F^{3/2}\eta^{3/2})$ based on the following update rule:
\begin{equation*}
\begin{split}
    \ub_{t+1} &= \argmin_{\ub\in\mathcal{X}^{s\prime}} \left\{\nabla \Tilde{f}_t(\ub_t)^{\top}\ub + \frac{1}{\eta} \mathcal{D}_{\phi}(\ub,\ub_t)\right\},\\
    \xb_{t+1} &= \argmin_{\xb\in\mathcal{X}^{s}} \left\{\nabla {f}_t(\xb_t)^{\top}\xb + \frac{1}{2\eta} \norm{\xb - \xb_t}^2\right\}.
\end{split}   
\end{equation*}
\end{lemma}

\begin{theorem}[Theorem 7 in \cite{ghai2022non}]\label{T: Theorem 7 in ghai2022non}
    Given a convex and compact domain $\mathcal{X}\subset \mathcal{X}^s$, and not necessarily convex loss $f_t(\cdot)$ satisfying Assumption \ref{A: Gradient and diameter bounds}. When Assumption \ref{A: Properties of mapping q} is met, there exists an OMD object with convex loss $\Tilde{f}_t(\cdot)$, a convex domain and a strongly convex regularization $\phi$ satisfying Assumption \ref{A: Geometric relation between phi and q}.
\end{theorem}

\begin{lemma}\label{L: Update deviation between OGD and OMD (distributed)}
    Suppose Assumptions \ref{A: Geometric relation between phi and q} to \ref{A: Gradient and diameter bounds} hold and $\ub_{i,t} = q(\xb_{i,t}),\;\forall i\in[m]$, then 
    $$\norm{q(\xb_{i,t+1}) - \ub_{i,t+1}}= O(\eta^{3/2})$$
    based on the following update rule:

\begin{equation*}
\begin{split}
    \zb_{i,t} &= \argmin_{\ub\in\widehat{X}^{s\prime}}  \left\{\nabla \Tilde{f}_{i,t}(\ub_{i,t})^{\top}\ub + \frac{1}{\eta} \mathcal{D}_{\phi}(\ub,\ub_{i,t}) \right\},\\
    \ub_{i,t+1} &= \sum_j [\Pb^{}]_{ji} \zb_{j,t},\\
    \yb_{i, t} &= \argmin_{\xb\in\widehat{X}^{s}} \left\{ \nabla f_{i,t}(\xb_{i,t})^{\top}\xb + \frac{1}{2\eta} \norm{\xb - \xb_{i,t}}^2\right\},\\
    \xb_{i,t+1} &= \sum_j [\Pb^{}]_{ji} \yb_{j,t}.
\end{split}    
\end{equation*}
\end{lemma}
\begin{proof}
    We first upper bound $\norm{q(\xb_{i,t+1}) - \ub_{i,t+1}}$ as follows
    \begin{equation}\label{Eq1: Update deviation between OGD and OMD (distributed)}
    \begin{split}
        &\norm{q(\xb_{i,t+1}) - \ub_{i,t+1}}\\
        \leq &\norm{\sum_j [\Pb^{}]_{ji} \zb_{j,t} - \sum_j [\Pb^{}]_{ji} q(\yb_{j,t})} + \norm{\sum_j [\Pb^{}]_{ji} q(\yb_{j,t}) - q(\sum_j [\Pb^{}]_{ji} \yb_{j,t})}.
    \end{split}    
    \end{equation}
    To bound $\norm{\sum_j [\Pb^{}]_{ji} q(\yb_{j,t}) - q(\sum_j [\Pb^{}]_{ji} \yb_{j,t})}$, we consider the Taylor expansion of $q(\yb)$ w.r.t. a point $\hat{\yb}$ in the convex hull of $\{\yb_{i,t}\}_i$:
    \begin{equation}\label{Eq2: Update deviation between OGD and OMD (distributed)}
    \resizebox{.9\hsize}{!}
    {$   
        \begin{aligned}
            &\norm{\sum_j [\Pb^{}]_{ji} q(\yb_{j,t}) - q(\sum_j [\Pb^{}]_{ji} \yb_{j,t})}\\
            \leq &\norm{\sum_j [\Pb^{}]_{ji}\bigg(q(\hat{\yb}) + J_{q}(\hat{\yb})(\yb_{j,t} - \hat{\yb}) + O\big((\yb_{j,t} - \hat{\yb})^2\big)\bigg) - \bigg(q(\hat{\yb}) + J_{q}(\hat{\yb})(\sum_j [\Pb^{}]_{ji} \yb_{j,t} - \hat{\yb}) + O\big((\sum_j [\Pb^{}]_{ji} \yb_{j,t} - \hat{\yb})^2\big)\bigg)}\\
            \leq & \norm{\sum_j [\Pb]_{ji}O\big((\yb_{j,t} - \hat{\yb})^2} + \norm{O\big((\sum_j [\Pb^{}]_{ji} \yb_{j,t} - \hat{\yb})^2\big)}\\
            \leq &2\norm{D^2_{\{\yb_{i,t}\}_i}},
        \end{aligned}    
    $}
    \end{equation}
    where $D_{\{\yb_{i,t}\}}$ denotes the diameter of the convex hull of $\{\yb_{i,t}\}$ and is upper bounded as follows
    \begin{equation}\label{Eq3: Update deviation between OGD and OMD (distributed)}
    \begin{split}
        D_{\{\yb_{i,t}\}_i} &= \max_{(i,j)} \norm{\yb_{i,t} - \yb_{j,t}}\\
        &= \max_{(i,j)} \norm{\Pi_{\widehat{X}^s}\big(\xb_{i,t} - \eta\nabla f_{i,t}(\xb_{i,t})\big) - \Pi_{\widehat{X}^s}\big(\xb_{j,t} - \eta\nabla f_{j,t}(\xb_{j,t})\big)}\\
        &\leq \max_{(i,j)} \norm{\big(\xb_{i,t} - \eta\nabla f_{i,t}(\xb_{i,t})\big) - \big(\xb_{j,t} - \eta\nabla f_{j,t}(\xb_{j,t})\big)}\\
        &\leq \frac{4\sqrt{m}\beta G}{1-\beta}\eta + 2G\eta = O(\eta),
    \end{split}
    \end{equation}
    where the last inequality is based on Lemma \ref{L: consecutive distance2} and the Lipschitz continuity of the function sequence.

    Substituting Equations \eqref{Eq2: Update deviation between OGD and OMD (distributed)}, \eqref{Eq3: Update deviation between OGD and OMD (distributed)} into Equation \eqref{Eq1: Update deviation between OGD and OMD (distributed)} and based on Lemma \ref{L: Update deviation between OGD and OMD}, we have
    \begin{equation}\label{Eq4: Update deviation between OGD and OMD (distributed)}
    \begin{split}
        \norm{q(\xb_{i,t+1}) - \ub_{i,t+1}}\leq &\norm{\sum_j [\Pb^{}]_{ji} \zb_{j,t} - \sum_j [\Pb^{}]_{ji} q(\yb_{j,t})} + \norm{\sum_j [\Pb^{}]_{ji} q(\yb_{j,t}) - q(\sum_j [\Pb^{}]_{ji} \yb_{j,t})}\\
        \leq &O(W^4G_F^{3/2}\eta^{3/2}) + O(\eta^2) = O(\eta^{3/2}),
    \end{split}
    \end{equation}
    when $\eta$ is small enough.
\end{proof}

\begin{algorithm}[!h]
\caption{Distributed online mirror descent}
\label{alg:D-OMD}
\begin{algorithmic}[1]
    \STATE {\bf Distributed online mirror descent:}
    \STATE Let $T_s\triangleq(T_0+T_1 + D_G + 1)$.
    \FOR{$t=T_s,\ldots,T$}
        \FOR{$i=1,2,\ldots,m$}
            \STATE{\color{black} \begin{equation*}
                \zb_{i,t} = \argmax_{\zb \in \widehat{\mathcal{X}}^{s\prime}} \nabla \Tilde{f}_{i,t}(\ub_{i,t})^{\top}(\zb-\ub_{i,t}) + \frac{1}{\eta} \mathcal{D}_{\phi}(\zb,\ub_{i,t}),         \left(\Tilde{f}_{i,t}(\ub) = \Tilde{f}_{i,t}\big(q(\xb)\big),\; \ub_{i,t} = q(\xb_{i,t}).\right) 
            \end{equation*}}
        \ENDFOR
        \STATE For all $i\in[m]$, \begin{equation*}
                \ub^{\prime}_{i,t+1} = 
                \sum\limits_{j=1}^m[\Pb^{}]_{ji}\zb_{j,t}.
            \end{equation*}
    \ENDFOR
\end{algorithmic}
\end{algorithm}

\begin{proof}[Proof of Theorem \ref{T: Dynamic regret bound (non-convex)}]
As the proof of Theorem \ref{T: Dynamic regret bound}, we decompose the individual regret into three terms:
\begin{equation}\label{Eq1: Dynamic regret bound (non-convex)}
    \sum_{t}\sum_i f_{i,t}(\xb_{j,t}) - \sum_t f_t(\xb^*_t) = \underbrace{\sum_{t=1}^{T_s-1}\sum_i f_{i,t}(\xb_{j,t}) - f_{i,t}(\xb^*_t)}_{\text{Term I}} + \underbrace{\sum_{t=T_s}^{T}\sum_i f_{i,t}(\xb_{j,t}) - f_{i,t}(\Tilde{\xb}^*_t)}_{\text{Term II}} + \underbrace{\sum_{t=T_s}^{T} f_{t}(\Tilde{\xb}^*_t) - f_{t}(\xb^*_t)}_{\text{Term III}},
\end{equation}
where $\Tilde{\xb}^*_t$ is the projection of $\xb^*_t$ on $\mathcal{X}^s_{\textnormal{in}}$, which is a subset of $\widehat{\mathcal{X}}^s$ with $\tau_{\textnormal{in}} = 2\mathcal{B}_rL$ based on Equation \eqref{Eq1: Similarity property of different agents' feasible sets}. (By applying the max-consensus step with a finite number of iterations, $\widehat{\mathcal{X}}^s$ is constructed using the estimate with the maximum norm over agents' estimates.)

{\bf The upper bound of Term I}:\\
Likewise, during the estimation phase, $\gamma$ is chosen to be less than $\frac{\Delta^s}{LL_A}$ to ensure the safeness of each agent's action, and based on the Lipschitz property we have
\begin{equation}\label{Eq2: Dynamic regret bound (non-convex)}
\resizebox{\hsize}{!}
    {$   
    \sum_{t=1}^{T_s-1}\sum_i f_{i,t}(\xb_{j,t}) - f_{i,t}(\xb^*_t) = \sum_{t=1}^{T_s-1}\sum_i \Tilde{f}_{i,t}\big(q(\xb_{j,t})\big) - \Tilde{f}_{i,t}\big(q(\xb^*_t)\big) \leq \sum_{t=1}^{T_s-1}\sum_i G_F W\norm{\xb_{j,t} - \xb^*_t} \leq 2G_F WLm(T_0+T_1+D_G).
    $}
\end{equation}
{\bf The upper bound of Term II}:\\
Define $\widehat{\mathcal{X}}^{s\prime}\triangleq\{q(\xb)| \xb\in \widehat{\mathcal{X}}^{s}\}$, (same for $\mathcal{X}^s_{\text{in}}$ and $\mathcal{X}^s$). Then for any $q(\Tilde{\xb}^*_t)=\Tilde{\ub}^*_t\in \mathcal{X}^{s\prime}_{\textnormal{in}}$, based on Algorithm \ref{alg:D-OMD} we have
\begin{equation}\label{Eq3: Dynamic regret bound (non-convex)}
\begin{split}
    \eta \left(f_{i,t}(\xb_{i,t}) - f_{i,t}(\Tilde{\xb}^*_t)\right)&=\eta \left(\Tilde{f}_{i,t}(\ub_{i,t}) - \Tilde{f}_{i,t}(\Tilde{\ub}^*_t)\right)\\
    &\leq \eta \nabla \Tilde{f}_{i,t}(\ub_{i,t})^{\top}(\ub_{i,t} - \Tilde{\ub}^*_t)\\
    &= \left(\nabla \phi(\ub_{i,t}) - \nabla \phi(\zb_{i,t}) - \eta \nabla \Tilde{f}_{i,t}(\ub_{i,t})\right)^{\top}(\Tilde{\ub}^*_t - \zb_{i,t})\\
    &+ \left(\nabla \phi(\zb_{i,t}) - \nabla \phi(\ub_{i,t})\right)^{\top}(\Tilde{\ub}^*_t - \zb_{i,t}) + \eta \nabla\Tilde{f}_{i,t}(\ub_{i,t})^{\top}(\ub_{i,t} - \zb_{i,t})\\
    &\leq \left(\nabla \phi(\zb_{i,t}) - \nabla \phi(\ub_{i,t})\right)^{\top}(\Tilde{\ub}^*_t - \zb_{i,t}) + \eta \nabla\Tilde{f}_{i,t}(\ub_{i,t})^{\top}(\ub_{i,t} - \zb_{i,t})\\
    &=\mathcal{D}_{\phi}(\Tilde{\ub}^*_t, \ub_{i,t}) - \mathcal{D}_{\phi}(\Tilde{\ub}^*_t, \zb_{i,t}) - \mathcal{D}_{\phi}(\zb_{i,t}, \ub_{i,t}) + \eta \nabla\Tilde{f}_{i,t}(\ub_{i,t})^{\top}(\ub_{i,t} - \zb_{i,t})\\
    &\leq \mathcal{D}_{\phi}(\Tilde{\ub}^*_t, \ub_{i,t}) - \mathcal{D}_{\phi}(\Tilde{\ub}^*_t, \zb_{i,t}) - \mathcal{D}_{\phi}(\zb_{i,t}, \ub_{i,t}) + \frac{1}{2}\norm{\ub_{i,t} - \zb_{i,t}}^2 + \frac{\eta^2}{2}\norm{\nabla\Tilde{f}_{i,t}(\ub_{i,t})}^2\\
    &\leq \mathcal{D}_{\phi}(\Tilde{\ub}^*_t, \ub_{i,t}) - \mathcal{D}_{\phi}(\Tilde{\ub}^*_t, \zb_{i,t}) + \frac{\eta^2}{2}\norm{\nabla\Tilde{f}_{i,t}(\ub_{i,t})}^2\\
    & =  \mathcal{D}_{\phi}(\Tilde{\ub}^*_t, \ub_{i,t}) - \mathcal{D}_{\phi}(\Tilde{\ub}^*_t, \ub_{i,t}^{\prime}) + \mathcal{D}_{\phi}(\Tilde{\ub}^*_t, \ub_{i,t}^{\prime}) - \mathcal{D}_{\phi}(\Tilde{\ub}^*_t, \zb_{i,t}) + \frac{\eta^2}{2}\norm{\nabla\Tilde{f}_{i,t}(\ub_{i,t})}^2\\
    &\leq \mathcal{D}_{\phi}(\Tilde{\ub}^*_t, \ub_{i,t}) - \mathcal{D}_{\phi}(\Tilde{\ub}^*_t, \ub_{i,t}^{\prime}) + \sum_{j} [\Pb]_{ji}\mathcal{D}_{\phi}(\Tilde{\ub}^*_t, \zb_{j,t-1}) - \mathcal{D}_{\phi}(\Tilde{\ub}^*_t, \zb_{i,t}) + \frac{\eta^2}{2}\norm{\nabla\Tilde{f}_{i,t}(\ub_{i,t})}^2,
    \end{split}    
\end{equation}
 where the second inequality is based on the optimality of $\zb_{i,t}$; the fourth inequality is due to the strong convexity of $\phi(\cdot)$ and the fifth inequality is based on Assumption \ref{A: Jensen's inequality for the Bregman divergence}.\\\\  
Based on Theorem \ref{T: Theorem 7 in ghai2022non}, Lemma \ref{L: Update deviation between OGD and OMD (distributed)}, and the Lipschitz assumption on $\mathcal{D}_{\phi}$, we have
\begin{equation}\label{Eq4: Dynamic regret bound (non-convex)}
 \begin{split}
     \norm{\mathcal{D}_{\phi}(\Tilde{\ub}^*_t, \ub_{i,t}) - \mathcal{D}_{\phi}(\Tilde{\ub}^*_t, \ub_{i,t}^{\prime})}&\leq W\norm{\ub_{i,t} - \ub^{\prime}_{i,t}} \leq O(W\eta^{3/2}).
 \end{split}   
\end{equation}
And based on Lemma \ref{L: consecutive distance2}, we get
\begin{equation}\label{Eq5: Dynamic regret bound (non-convex)}
\begin{split}
    \max_{i,j\in [m]}\norm{\ub_{i,t} - \ub_{j,t}} = \max_{i,j\in [m]}\norm{q(\xb_{i,t}) - q(\xb_{j,t})} = O\big(W\eta\big).
\end{split} 
\end{equation}
With Equations \eqref{Eq3: Dynamic regret bound (non-convex)}, \eqref{Eq4: Dynamic regret bound (non-convex)} and \eqref{Eq5: Dynamic regret bound (non-convex)}, we derive
\begin{equation}\label{Eq6: Dynamic regret bound (non-convex)}
\begin{split}
    \Tilde{f}_{i,t}(\ub_{j,t}) - \Tilde{f}_{i,t}(\Tilde{\ub}^*_t) &= \Tilde{f}_{i,t}(\ub_{j,t}) - \Tilde{f}_{i,t}(\ub_{i,t}) + \Tilde{f}_{i,t}(\ub_{i,t}) - \Tilde{f}_{i,t}(\Tilde{\ub}^*_t)\\
    &\leq G_F\norm{\ub_{i,t} - \ub_{j,t}} + O(\eta^{1/2}W)\\
    &+ \frac{1}{\eta}\sum_{j} [\Pb]_{ji}\mathcal{D}_{\phi}(\Tilde{\ub}^*_t, \zb_{j,t-1}) - \frac{1}{\eta}\mathcal{D}_{\phi}(\Tilde{\ub}^*_t, \zb_{i,t}) + \frac{\eta}{2}\norm{\nabla\Tilde{f}_{i,t}(\ub_{i,t})}^2\\
    &\leq O\big(G_FW\eta\big) + O(\eta^{1/2}W)\\
    &+ \frac{1}{\eta}\sum_{j} [\Pb]_{ji}\mathcal{D}_{\phi}(\Tilde{\ub}^*_t, \zb_{j,t-1}) - \frac{1}{\eta}\mathcal{D}_{\phi}(\Tilde{\ub}^*_t, \zb_{i,t}) + \frac{\eta}{2}\norm{\nabla\Tilde{f}_{i,t}(\ub_{i,t})}^2.
\end{split}    
\end{equation}
Based on the definition of Bregman divergence, we have the following relationship
\begin{equation}\label{Eq7: Dynamic regret bound (non-convex)}
\begin{split}
    &\mathcal{D}_{\phi}(\Tilde{\ub}^*_t, \zb_{i,t-1}) - \mathcal{D}_{\phi}(\Tilde{\ub}^*_t, \zb_{i,t})\\
    =&\left(\nabla \phi(\zb_{i,t}) - \nabla \phi(\zb_{i,t-1})\right)^{\top}(\Tilde{\ub}^*_t - \zb_{i,t}) + \mathcal{D}_{\phi}(\zb_{i,t} - \zb_{i,t-1})\\
    =&\left(\nabla \phi(\zb_{i,t}) - \nabla \phi(\zb_{i,t-1})\right)^{\top}\Tilde{\ub}^*_t + \left(\phi(\zb_{i,t}) - \nabla \phi(\zb_{i,t})^{\top}\zb_{i,t}\right) - \left(\phi(\zb_{i,t-1}) - \nabla \phi(\zb_{i,t-1})^{\top}\zb_{i,t-1}\right)
\end{split}    
\end{equation}
Summing Equation \eqref{Eq6: Dynamic regret bound (non-convex)} over $i$, then based on Equation \eqref{Eq7: Dynamic regret bound (non-convex)} we get
\begin{equation}\label{Eq8: Dynamic regret bound (non-convex)}
\begin{split}
    &\sum_{i}\Tilde{f}_{i,t}(\ub_{j,t}) - \Tilde{f}_{i,t}(\Tilde{\ub}^*_t)\\
    \leq &O\big(mG_FW\eta\big) + O(m\eta^{1/2}W) + \sum_i\frac{\eta}{2}\norm{\nabla\Tilde{f}_{i,t}(\ub_{i,t})}^2\\
    + &\frac{1}{\eta}\sum_i \left[\left(\nabla \phi(\zb_{i,t}) - \nabla \phi(\zb_{i,t-1})\right)^{\top}\Tilde{\ub}^*_t + \left(\phi(\zb_{i,t}) - \nabla \phi(\zb_{i,t})^{\top}\zb_{i,t}\right) - \left(\phi(\zb_{i,t-1}) - \nabla \phi(\zb_{i,t-1})^{\top}\zb_{i,t-1}\right)\right].
\end{split}
\end{equation}

Then by summing Equation \eqref{Eq8: Dynamic regret bound (non-convex)} over $[T_s,\ldots, T]$, we have
\begin{equation}\label{Eq9: Dynamic regret bound (non-convex)}
\begin{split}
    &\sum_{t=T_s}^T\sum_{i}\Tilde{f}_{i,t}(\ub_{j,t}) - \Tilde{f}_{i,t}(\Tilde{\ub}^*_t)\\
    \leq &O\big(mTG_FW\eta\big) + O(mT\eta^{1/2}W) + \sum_{t=T_s}^T\sum_i\frac{\eta}{2}\norm{\nabla\Tilde{f}_{i,t}(\ub_{i,t})}^2\\
    + &\frac{1}{\eta}\left[ \sum_{t=T_s}^T (\Tilde{\ub}^*_t - \Tilde{\ub}^*_{t+1})^{\top}\left(\sum_i \nabla \phi(\zb_{i,t})\right) + \left(\sum_i \nabla \phi(\zb_{i,T})\right)^{\top}(\Tilde{\ub}^*_T) - \left(\sum_i \nabla \phi(\zb_{i,T_s-1})\right)^{\top}(\Tilde{\ub}^*_{T_s}) \right]\\
    + &\frac{1}{\eta}\sum_i\left[\left(\phi(\zb_{i,T}) - \nabla \phi(\zb_{i,T})^{\top}\zb_{i,T}\right) - \left(\phi(\zb_{i,T_s-1}) - \nabla \phi(\zb_{i,T_s-1})^{\top}\zb_{i,T_s-1}\right)\right].
\end{split}    
\end{equation}
{\bf The upper bound of Term III}:\\
Based on Lemma \ref{lem:fereydounian2020safe-lemma}, we have for any $\xb^*_t\in \mathcal{X}^s$ and its projection to $\mathcal{X}^s_{\textnormal{in}}$: $\Tilde{\xb}^*_t$
\begin{equation}\label{Eq10: Dynamic regret bound (non-convex)}
\begin{split}
    \sum_{t=T_s}^T\sum_i \left(\Tilde{f}_{i,t}(q(\Tilde{\xb}^*_t)) - \Tilde{f}_{i,t}(q(\xb^*_t))\right) \leq \sum_{t=T_s}^T\sum_i G_FW\norm{\Tilde{\xb}^*_t - \xb^*_t} \leq mT G_FW \frac{2\sqrt{d}L\mathcal{B}_r}{C(\Ab, \bb)}.
\end{split}
\end{equation}
Substituting Equations \eqref{Eq2: Dynamic regret bound (non-convex)}, \eqref{Eq9: Dynamic regret bound (non-convex)} and \eqref{Eq10: Dynamic regret bound (non-convex)} into Equation \eqref{Eq1: Dynamic regret bound (non-convex)}, the final regret bound is as
\begin{equation}\label{Eq11: Dynamic regret bound (non-convex)}
\begin{split}
    &\sum_{t=1}\sum_i \left(f_{i,t}(\xb_{j,t}) - f_{i,t}(\xb^*_t)\right)\\
    \leq &O\big(mTG_FW\eta\big) + O(mT\eta^{1/2} W) + \sum_{t=T_s}^T\sum_i\frac{\eta}{2}\norm{\nabla\Tilde{f}_{i,t}(\ub_{i,t})}^2\\
    + &\frac{1}{\eta}\left[ \sum_{t=T_s}^T (\Tilde{\ub}^*_t - \Tilde{\ub}^*_{t+1})^{\top}\left(\sum_i \nabla \phi(\zb_{i,t})\right) + \left(\sum_i \nabla \phi(\zb_{i,T})\right)^{\top}(\Tilde{\ub}^*_T) - \left(\sum_i \nabla \phi(\zb_{i,T_s-1})\right)^{\top}(\Tilde{\ub}^*{T_s}) \right]\\
    + &\frac{1}{\eta}\sum_i\left[\left(\phi(\zb_{i,T}) - \nabla \phi(\zb_{i,T})^{\top}\zb_{i,T}\right) - \left(\phi(\zb_{i,T_s-1}) - \nabla \phi(\zb_{i,T_s-1})^{\top}\zb_{i,T_s-1}\right)\right] + O(T_0+T_1) + mT G_FW \frac{2\sqrt{d}L\mathcal{B}_r}{C(\Ab, \bb)}\\
    = &{\color{black} O(T_0 + T_1 + T\sqrt{\eta} + \frac{T\sqrt{\log T_0}}{\sqrt{T_0}} + \frac{1}{\eta} + \frac{1}{\eta}\sum_{t=T_s}^T \norm{\Tilde{\ub}^*_t - \Tilde{\ub}^*_{t+1}})},
\end{split}
\end{equation}
where the final regret bound is proved by applying the specified $\eta$ and $T_0$.
\end{proof}

\end{document}